\documentclass[a4paper,notitlepage,twoside,leqno,12pt]{amsart}

\usepackage{mathrsfs}
\usepackage{bbm,pifont,latexsym}
\usepackage{anysize}  %方便设定纸张与留空的大小
\marginsize{2.4cm}{2.4cm}{2.4cm}{2.4cm}

\usepackage{dcolumn,indentfirst,color}
\usepackage[pdfpagemode=UseNone,colorlinks=true,linkcolor=black,citecolor=black,pdfstartview=FitH]{hyperref}
\usepackage{amsmath,amssymb,amscd,amsthm,amsfonts,mathrsfs}
\usepackage{color,graphicx,xcolor,graphics}
\usepackage{titlesec} %设置标题格式
\usepackage{titletoc} %设置目录格式
\titlecontents{section}[20pt]{\addvspace{2pt}\filright}
              {\contentspush{\thecontentslabel. }}
              {}{\titlerule*[8pt]{}\contentspage}% 中间的括号用于填充

\usepackage{fancyhdr} %页眉和页脚
\newtheorem{thm}{Theorem}[section]
\newtheorem{lema}[thm]{Lemma}
\newtheorem{cor}[thm]{Corollary}

\theoremstyle{definition}
\newtheorem{rmk}[thm]{Remark}

\makeatletter\@addtoreset{equation}{section}\makeatother %每一节的方程重新计数

\titleformat{\section}{\centering\normalsize}{\textsc{\thesection.}}{0.5em}{\textsc}%0.5em表示章节数与title之间的距离
\titleformat{\subsection}[runin]{\normalsize}{\thesubsection.}{0.3em}{\textbf}

\begin{document}

\author{WEIYUAN QIU}
\address{Weiyuan Qiu, School of Mathematical Sciences, Fudan University, Shanghai, 200433, P. R. China}
\email{wyqiu@fudan.edu.cn}

\author{FEI YANG}
\address{Fei Yang, Department of Mathematics, Nanjing University, Nanjing, 210093, P. R. China}
\email{yangfei\rule[-2pt]{0.2cm}{0.5pt}math@163.com}

\author{YONGCHENG YIN}
\address{Yongcheng Yin, Department of Mathematics, Zhejiang University, Hangzhou, 310027, P. R. China}
\email{yin@zju.edu.cn}

%---------------------------------------------------------------------------------------------------------------
\title{RATIONAL MAPS WHOSE JULIA SETS ARE CANTOR CIRCLES}

\begin{abstract}
In this paper, we give a family of rational maps whose Julia sets are Cantor circles and show that every rational map whose Julia set is a Cantor set of circles must be topologically conjugate to one map in this family on their corresponding Julia sets. In particular, we give the specific expressions of some rational maps whose Julia sets are Cantor circles, but they are not topologically conjugate to any McMullen maps on their Julia sets. Moreover, some non-hyperbolic rational maps whose Julia sets are Cantor circles are also constructed.
\end{abstract}

% AMS subject classifications (used in AMS journals)
\subjclass[2010]{Primary 37F45; Secondary 37F20}

% AMS keywords (used in AMS journals)
\keywords{Julia sets, Cantor circles, rational maps}

% today's date, or fill in whatever date you prefer
\date{\today}

% acknowledge support, etc
% \thanks{This research was partially supported by NSF grant DOA-123456789.}
% \thanks{We would like to thank our colleagues for their helpful criticism.}

% dedication
% \dedicatory{Dedicated to Professor Donald Knuth on the occasion of his $100$th birthday}

\maketitle

%----------------------------------------------------------------------------------------------------------------
%\vskip1.0cm
%\tableofcontents
%----------------------------------------------------------------------------------------------------------------

%----------------------------------------------------------------------------------------------------------------
\section{Introduction}

The study of the topological properties of the Julia sets of rational maps is a central problem in complex dynamics. For each degree at least two polynomial with a disconnected Julia set, it was proved that all but countably many components of the Julia set are single points in \cite{QY}. For rational maps, the Julia sets may exhibit more complex topological structures. Pilgrim and Tan proved that if the Julia set of a hyperbolic (more generally, geometrically finite) rational map is disconnected, then, with the possible exception of finitely many periodic components and their countable collection of preimages, every Julia component is either a point or a Jordan curve \cite[Theorem 1.2]{PT}. In this paper, we will consider one class of rational maps whose Julia sets possess simple topological structure: each Julia component is a Jordan curve.

A subset of the Riemann sphere $\overline{\mathbb{C}}$ is called a \textit{Cantor set of circles} (sometimes \textit{Cantor circles} in short) if it consists of uncountably many closed Jordan curves which is homeomorphic to $\mathcal{C}\times \mathbb{S}^1$, where $\mathcal {C}$ is the middle third Cantor set and $\mathbb{S}^1$ is the unit circle. The first example of rational map whose Julia set is a Cantor set of circles was discovered by McMullen (see \cite[$\S$7]{Mc}). He showed that if $f(z)=z^2+\lambda/z^3$ and $\lambda$ is small enough, then the Julia set of $f$ is a Cantor set of circles. Later, many authors focus on the following family, which is commonly referred as the \emph{McMullen maps}:
\begin{equation}\label{McMullen}
g_{\eta}(z)=z^k+\eta/z^l,
\end{equation}
where $k,l\geq 2$ and $\eta\in\mathbb{C}\setminus\{0\}$ (see \cite{DLU,St,QWY} and the references therein). These special rational maps can be viewed as a perturbation of the simple polynomial $g_0(z)=z^k$ if $\eta$ is small. It is known that when $1/k+1/l<1$, there exists a punched neighborhood $\mathcal{M}$ centered at origin in the parameter space, which is called the \textit{McMullen domain}, such that when $\eta\in\mathcal{M}$, then the Julia set of $g_\eta$ is a Cantor set of circles (see \cite[$\S$7]{Mc} for $k=2,l=3$ and \cite[$\S$3]{DLU} for the general cases).

The following three questions arise naturally: (1) Besides McMullen maps, do there exist any other rational maps whose Julia sets are Cantor circles? (2) If the answer to the first question is yes, what do they look like? Or in other words, can we find specific expressions for them? (3) Can we find out all rational maps whose Julia sets are Cantor circles in some sense? This paper will give affirmative answers to these questions.

By quasiconformal surgery, we can obtain many new rational maps after perturbing the immediate super-attracting basin centered at $\infty$ of $g_\eta$ into a geometric one. Fix one of them, then this map is not topologically conjugate to $g_\eta$ on the whole $\overline{\mathbb{C}}$. But they are topologically conjugate to each other on their corresponding Julia sets. In particular, $h_{c,\eta}(z)=\frac{1}{z}\circ (z^k+c)\circ\frac{1}{z}+\eta/z^l$ is an example, where $1/k+1/l<1$ and $c,\eta\in\mathbb{C}\setminus\{0\}$ are both small enough. However, these types of rational maps can be also regarded as the McMullen maps essentially, which are not what we want to find since they can be obtained by doing a surgery only on the Fatou sets of the genuine McMullen maps.
So it will be very interesting to find other types of rational maps with Cantor circles Julia sets which are not topologically conjugate to any McMullen maps on their corresponding Julia sets.

The existence of of types of rational maps `essentially' different from McMullen maps was known previously (see \cite[$\S\S$1,2]{HP}). Here, `essentially' means there exists no topological conjugacy between the Julia sets of McMullen maps and the rational maps whose Julia sets are Cantor circles. In this paper, we will give the specific expressions for these types of rational maps, not only including the cases discussed in \cite{HP}, but also covering all the rational maps whose Julia sets are Cantor circles `essentially' (see Theorem \ref{this-is-all}).

Let $p\in\{0,1\}$, $n\geq 2$ be an integer and $d_1,\cdots,d_n$ be $n$ positive integers such that $\sum_{i=1}^{n}(1/d_i)<1$. We define
\begin{equation}\label{family}
f_{p,d_1,\cdots,d_n}(z)=z^{(-1)^{n-p} d_1}\prod_{i=1}^{n-1}(z^{d_i+d_{i+1}}-a_i^{d_i+d_{i+1}})^{(-1)^{n-i-p}},
\end{equation}
where $a_1,\cdots,a_{n-1}$ are $n-1$ small complex numbers satisfying $0<|a_1|<\cdots<|a_{n-1}|<1$. In particular, if $n=2$, then $f_{1,d_1,d_2}(z)=z^{d_2}-a_1^{d_1+d_2}/z^{d_1}$ is the McMullen map that has been well studied by many authors. Moreover, $f_{0,d_1,d_2}(z)=z^{d_1}/(z^{d_1+d_2}-a_1^{d_1+d_2})$ is conformally conjugate to the McMullen map $z\mapsto z^{d_1}+\eta/z^{d_2}$ for some $\eta\neq 0$. The degrees of $f_{p,d_1,\cdots,d_n}$ at $0$ and $\infty$ are $d_1$ and $d_n$ respectively and  $\text{deg} (f_{p,d_1,\cdots,d_n})=\sum_{i=1}^{n}d_i$. For each element in the family \eqref{family}, it is easy to check that $0$ and $\infty$ belong to the Fatou set of $f_{p,d_1,\cdots,d_n}$. Let $D_0$ and $D_\infty$ be the Fatou components containing $0$ and $\infty$ respectively. There are four cases (we use $f$ to replace $f_{p,d_1,\cdots,d_n}$ temporarily):

(1) If $p=1$ and $n$ is odd, then $f(D_0)=D_0$ and $f(D_\infty)=D_\infty$;

(2) If $p=1$ and $n$ is even, then $f(D_0)=D_\infty$ and $f(D_\infty)=D_\infty$;

(3) If $p=0$ and $n$ is odd, then $f(D_0)=D_\infty$ and $f(D_\infty)=D_0$;

(4) If $p=0$ and $n$ is even, then $f(D_0)=D_0$ and $f(D_\infty)=D_0$.

Firstly we will find suitable parameters $a_i$ in \eqref{family}, where $1\leq i\leq n-1$, such the Julia set of each $f_{p,d_1,\cdots,d_n}$ in the four cases stated above is a Cantor set of circles.

\begin{thm}\label{parameter}
For each given $p\in\{0,1\}$, $n\geq 2$ and $d_1,\cdots,d_n$ satisfying $\sum_{i=1}^{n}(1/d_i)<1$, there exist suitable parameters $a_i$, where $1\leq i\leq n-1$ such that the Julia set of $f_{p,d_1,\cdots,d_n}$ is a Cantor set of circles.
\end{thm}

The specific value ranges of $a_i$ are given in \S\ref{sec-loc-crit}, where $1\leq i\leq n-1$ (see \eqref{range-s1}, \eqref{range-s0} and Theorem \ref{parameter-restate}). These rational maps can be seen as the perturbations of $z^{d_n}$ or $z^{-d_n}$ (according to whether $p=1$ or 0) since each $a_i$ can be arbitrarily small (see Theorem \ref{parameter-restate}). Moreover, it will be shown that if $n\geq 3$, then each $f_{p,d_1,\cdots,d_n}$ is not topologically conjugate to any McMullen maps on their corresponding Julia sets (see Theorem \ref{no-topo-equiv}). This means that we have found the specific expressions of rational maps whose Julia sets are Cantor circles which are `essentially' different from McMullen maps.

For example, let $p=1$, $n=4$, $d_1=d_2=d_3=d_4=5$ and define
\begin{equation}\label{an-example}
f_{1,5,5,5,5}(z)=\frac{(z^{10}-a_1^{10})(z^{10}-a_3^{10})}{z^5(z^{10}-a_2^{10})},
\end{equation}
where $a_1=0.00025,a_2=0.005$ and $a_3=0.1$. By a straightforward calculation or using Theorem \ref{parameter-restate} and Remark \ref{range-unif}, one can show that the Julia set of $f_{1,5,5,5,5}$ is a Cantor set of circles (see Figure \ref{Fig_Cantor-cicle}). The dynamics on the set of Julia components of $f_{1,5,5,5,5}$ is conjugate to the one-sided shift on four symbols $\Sigma_4:=\{0,1,2,3\}^{\mathbb{N}}$ while the set of Julia components of $g_\eta$ is conjugate to the one-sided shift on only two symbols $\Sigma_{2}:=\{0,1\}^{\mathbb{N}}$. This means that $f_{1,5,5,5,5}$ cannot be topologically conjugate to $g_\eta$ on their corresponding Julia sets.

\begin{figure}[!htpb]
  \setlength{\unitlength}{1mm}
  \centering
  \includegraphics[width=65mm]{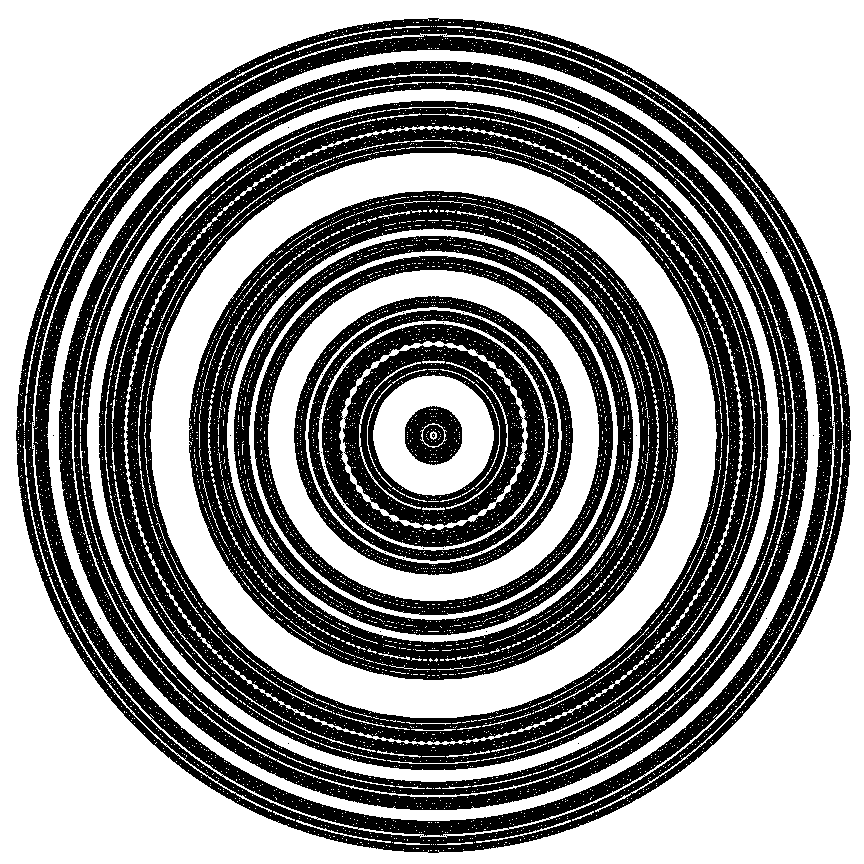}
  \includegraphics[width=65mm]{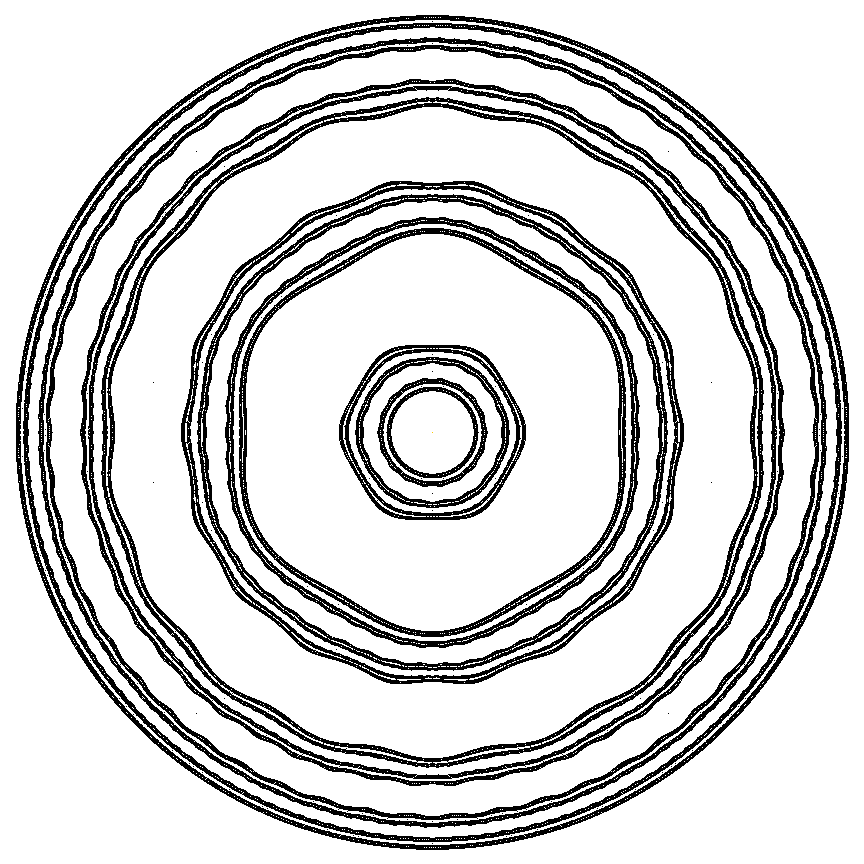}
  \caption{The Julia set of $f_{1,5,5,5,5}$ (left picture), which is not topologically conjugate to that of McMullen map $g_\eta(z)=z^3+0.001/z^3$ (right picture). The two Julia sets are both Cantor circles.}
  \label{Fig_Cantor-cicle}
\end{figure}

Note that if the Julia set $J(f)$ of a rational map $f$ is a Cantor set of circles, then there exist no critical points in $J(f)$ since each Julia component is a Jordan closed curve (see Lemma \ref{no-crit-on-J}). This means that every periodic Fatou component of $f$ must be attracting or parabolic. In fact, we have following theorem.

\begin{thm}\label{this-is-all}
Let $f$ be a rational map whose Julia set is a Cantor set of circles. Then there exist $p\in\{0,1\}$, positive integers $n\geq 2$, and $d_1,\cdots, d_n$ satisfying $\sum_{i=1}^{n}(1/d_i)<1$ such that $f$ is topologically conjugate to $f_{p,d_1,\cdots,d_n}$ on their corresponding Julia sets for suitable parameters $a_i$, where $1\leq i\leq n-1$.
\end{thm}

Since the dynamics on the Fatou set can be perturbed freely, it follows from Theorem \ref{this-is-all} that we have found `all' the possible rational maps whose Julia sets are Cantor circles. A rational map is \textit{hyperbolic} if all critical points are attracted by attracting periodic orbits. For the regularity of the Julia components of $f_{p,d_1,\cdots, d_n}$, it can be shown that each Julia component of $f_{p,d_1,\cdots,d_n}$ is a quasicircle if $f_{p,d_1,\cdots,d_n}$ is hyperbolic (see Corollary \ref{Julia-comp}).

If $\eta$ is small enough, then $g_\eta$ is hyperbolic (see \cite{DLU}). Now we construct some non-hyperbolic rational maps whose Julia sets are Cantor circles. Let $m,n\geq 2$ be two positive integers satisfying $1/m+1/n<1$ and $\lambda\in\mathbb{C}\setminus\{0\}$, we define
\begin{equation}
P_\lambda(z)=\frac{\frac{1}{n}((1+z)^n-1)+\lambda^{m+n}z^{m+n}}{1-\lambda^{m+n}z^{m+n}}.
\end{equation}
It is straightforward to verify that zero is a parabolic fixed point of $P_\lambda$ with multiplier one. We then have the following theorem.

\begin{thm}\label{non-hyper-cantor}
If $0<|\lambda|\leq 1/{(2^{10m}n^3)}$, then $P_\lambda$ is non-hyperbolic and its Julia set is a Cantor set of circles.
\end{thm}

Inspired by Theorem \ref{parameter}, we can construct more non-hyperbolic rational maps whose Julia sets are Cantor circles. For simplicity, for each $n\geq 2$, we only consider the case $d_i=n+1$ for every $1\leq i\leq n$. For every $n\geq 2$, we define
\begin{equation}\label{family-para}
P_n(z)=A_n\,\frac{(n+1)z^{(-1)^{n+1} (n+1)}}{nz^{n+1}+1}\prod_{i=1}^{n-1}(z^{2n+2}-b_i^{2n+2})^{(-1)^{i-1}}+B_n,
\end{equation}
where $b_1,\cdots,b_{n-1}$ are $n-1$ small complex numbers satisfying $1>|b_1|>\cdots>|b_{n-1}|>0$ and
\begin{equation}\label{A-B-n}
\begin{split}
A_n=\frac{1}{1+(2n+2)C_n}\prod_{i=1}^{n-1}(1-b_i^{2n+2})^{(-1)^i},  &~~B_n=\frac{(2n+2)C_n}{1+(2n+2)C_n} \\
\text{and}~C_n=\sum_{i=1}^{n-1}\frac{(-1)^{i-1}b_i^{2n+2}}{1-b_i^{2n+2}}.  &
\end{split}
\end{equation}

The terms $A_n$ and $B_n$ here can guarantee that $P_n(1)=1$ and $P_n'(1)=1$. Namely, $1$ is a parabolic fixed point of $P_n$ with multiplier one (see Lemma \ref{para-fixed}).

\begin{thm}\label{parameter-parabolic}
For every $n\geq 2$ and $1\leq i\leq n-1$, if $|b_i|=s^i$ for $0<s\leq 1/(25n^2)$, then $P_n$ is non-hyperbolic and its Julia set is a Cantor set of circles.
\end{thm}

It can be seen later the dynamics of $P_n$ on their Julia sets are conjugate to that of $f_{1,n+1,\cdots,n+1}$ for $n\geq 2$. One of the differences between their dynamics on the Fatou sets is the super-attracting basin of $f_{1,n+1,\cdots,n+1}$ at $\infty$ is replaced by a parabolic basin of $P_n$.

This paper is organized as follows: In \S\ref{sec-loc-crit}, we do some estimates and prove Theorem \ref{parameter}. In \S\ref{sec-topo-conj}, we prove Theorem \ref{this-is-all}. In \S\ref{sec-para-mcm}, we show that the Julia set of $P_\lambda$ is a Cantor set of circles if $\lambda$ is small enough and prove Theorem \ref{non-hyper-cantor}. We will prove Theorem \ref{parameter-parabolic} in \S\ref{sec-more-exam} and leave a key lemma to the last section.

\vskip0.2cm
\noindent\textit{Notation}. We will use the following notations throughout the paper. Let $\mathbb{C}$ be the complex plane and $\overline{\mathbb{C}}=\mathbb{C}\cup\{\infty\}$ the Riemann sphere. For $r>0$ and $a\in\mathbb{C}$, let $\mathbb{D}(a,r):=\{z\in\mathbb{C}:|z-a|<r\}$ be the Euclidean disk centered at $a$ with radius $r$. In particular, let $\mathbb{D}_r:=\mathbb{D}(0,r)$ be the disk centered at the origin with radius $r$ and $\mathbb{T}_r:=\partial\mathbb{D}_r$ be the boundary of $\mathbb{D}_r$. As usual, $\mathbb{D}:=\mathbb{D}_1$ and $\mathbb{S}^1:=\mathbb{T}_1$ denote the unit disk and the unit circle, respectively. For $0<r<R<+\infty$, let $\mathbb{A}_{r,R}:=\{z\in\mathbb{C}:r<|z|<R\}$ be the round annulus centered at the origin.

\section{Location of the critical points and the hyperbolic case}\label{sec-loc-crit}

First we give some basic and useful estimations.

\begin{lema}\label{very-useful-est}
Let $n\geq 2$ be an integer, $a\in\mathbb{C}\setminus\{0\}$ and $0<\varepsilon<1/2$.

$(1)$ If $|z-a|\leq \varepsilon |a|$, then $|z^{n}-a^{n}|\leq ((1+\varepsilon)^{n}-1)\, |a|^{n}$;

$(2)$ If $|z^{n}-a^{n}|\leq \varepsilon |a|^{n}$, then $|a/z|^n<1+2\varepsilon$ and $|z-ae^{2\pi i{j}/{n}}|< \varepsilon |a|$ for some $1\leq j\leq n$;

$(3)$ If $0<\varepsilon<1/n$, then $n\varepsilon< (1+\varepsilon)^n-1< 3n\varepsilon$ and $n\varepsilon/3< 1-(1-\varepsilon)^n<n\varepsilon$.
\end{lema}

\begin{proof}
Let $z=a(1+re^{i\theta})$ for $0\leq r\leq \varepsilon$ and $0\leq\theta<2\pi$, then
\begin{equation*}
|z^{n}-a^{n}|= |(1+re^{i\theta})^{n}-1|\cdot|a|^{n}\leq ((1+\varepsilon)^{n}-1)\, |a|^{n}.
\end{equation*}
This proves (1). The first statement in (2) follows from $|a/z|^n\leq 1/(1-\varepsilon)<1+2\varepsilon$ if $0<\varepsilon<1/2$. For the second statement, let $z^{n}=a^{n}(1+re^{i\theta})$ for $0\leq r\leq \varepsilon$ and $0\leq\theta<2\pi$, then $z=ae^{2\pi i{j}/{n}}(1+re^{i\theta})^{1/n}$ for some $1\leq j\leq n$ and we have
\begin{equation*}
|z-ae^{2\pi i{j}/{n}}|=|(1+re^{i\theta})^{1/n}-1|\cdot|a| \leq ((1+\varepsilon)^{1/n}-1)\cdot|a| < \varepsilon |a|
\end{equation*}
if $n\geq 2$. The claim (3) can be proved by using Lagrange's mean value theorem to $x\mapsto x^n$ on the intervals $[1,1+\varepsilon]$ and $[1-\varepsilon,1]$ respectively. The proof is complete.
\end{proof}

Fix $n\geq 2$ and let $d_1,\cdots,d_n\geq 2$ be $n$ positive numbers such that $\xi=\sum_{i=1}^{n}(1/d_i)<1$. We use $K\geq 3$ to denote the maximal number among $d_1,\cdots,d_n$. Let $u_1=s_1 K^{-5}$ and $v_1=s_1 K^{-2}$, where
\begin{equation}\label{range-s1}
0<s_1\leq \min\{K^{-5\xi/(1-\xi)},K^{5-2K}\}<1.
\end{equation}
Let $u_0=s_0^{1+1/d_n+2(1-\xi)/3}$, $v_0=s_0^{1/d_n+(1-\xi)/3}$, where
\begin{equation}\label{range-s0}
0<s_0\leq \min\{2^{-(1-\xi)^{-1}(1+1/d_n-2\xi/3)^{-1}},(4K)^{-3/(1-\xi)},K^{-2K(1+1/d_n+2(1-\xi)/3)^{-1}}\}<1.
\end{equation}
For $p\in\{0,1\}$, let $|a_{n-1,p}|=v_p^{1/d_{n}}$ and $|a_{i,p}|=u_p^{1/d_{i+1}}|a_{i+1,p}|$ be the $n-1$ parameters in the family $f_{p,d_1,\cdots,d_n}$, where $1\leq i\leq n-2$.
Since the cases $p=0$ and $p=1$ can be discussed uniformly in general, we use $s$, $u$, $v$ and $a_i$, respectively, to denote $s_p$, $u_p$, $v_p$ and $a_{i,p}$ for simplicity when the situation is clear, where $1\leq i\leq n-1$.

\begin{lema}\label{esti-a1}
$(1)$ $u^{2/K}\leq K^{-4}$.

$(2)$ If $1\leq j\leq i\leq n-1$, then $|a_j/a_i|\leq u^{\frac{i-j}{K}}$.

$(3)$ If $p=1$, then

~~\textup{(3a)} $(s/|a_1|)^{d_1}< su/(2v)=sK^{-3}/2$ and

~~\textup{(3b)} $(|a_1|/s)^{d_1}v/2>K$.

$(4)$ If $p=0$, then

~~\textup{(4a)} $2Ku/v<s$ and $1/(2Kv)>(2/s)^{1/d_n}$;

~~\textup{(4b)} $(s/|a_1|)^{d_1}<sv/2<u^{1/2}/2$ and

~~\textup{(4c)} $(|a_1|/s)^{d_1}u/(2v)>(2/s)^{1/d_n}$.
\end{lema}
\begin{proof}
(1) From \eqref{range-s1} and \eqref{range-s0}, we have $s_1\leq K^{5-2K}$ and $s_0\leq K^{-2K(1+1/d_n+2(1-\xi)/3)^{-1}}$. This means that $u_1^{2/K}=(s_1 K^{-5})^{2/K}\leq K^{-4}$ and $u_0^{2/K}\leq K^{-4}$.

(2) If $j=i$, then (2) is trivial. Suppose that $1\leq j<i\leq n-1$, then
\begin{equation*}
|a_j/a_i|= u^{\frac{1}{d_{j+1}}+\cdots+\frac{1}{d_{i}}}\leq u^{\frac{i-j}{K}}
\end{equation*}
since $K\geq d_i$ for $1\leq i\leq n$. This proves (2).

(3) If $p=1$, then $u=s K^{-5}$ and $v=s K^{-2}$. Since $s\leq K^{-5\xi/(1-\xi)}$, we have $s^{1-\xi}K^{5\xi}\leq 1$, so
\begin{equation*}
s^{1-\frac{1}{d_1}} s^{-(\frac{1}{d_2}+\cdots+\frac{1}{d_n})} K^{5(\frac{1}{d_2}+\cdots+\frac{1}{d_{n-1}})+\frac{2}{d_n}} 2^{\frac{1}{d_1}} K^{\frac{3}{d_1}}<1.
\end{equation*}
This is equivalent to $s^{1-\frac{1}{d_1}}2^{\frac{1}{d_1}} K^{\frac{3}{d_1}}/|a_1|<1$ since
\begin{equation*}
|a_1|=u^{\frac{1}{d_2}+\cdots+\frac{1}{d_{n-1}}}v^{\frac{1}{d_{n}}}=s^{\frac{1}{d_2}+\cdots+\frac{1}{d_n}}/K^{5(\frac{1}{d_2}+\cdots+\frac{1}{d_{n-1}})+\frac{2}{d_n}}.
\end{equation*}
So we have $(s/|a_1|)^{d_1}< su/(2v)=sK^{-3}/2$ and \textup{(3a)} is proved. Moreover, (3b) can be derived from (3a) directly since $(|a_1|/s)^{d_1}>2K^3/s=2K/v$.

(4) If $p=0$, then $u=s^{1+1/d_n+2(1-\xi)/3}$, $v=s^{1/d_n+(1-\xi)/3}$. From \eqref{range-s0}, we know $4Ks^{(1-\xi)/3}\leq 1$, which means $2Ku/v=2Ks^{1+(1-\xi)/3}<s$. Note that $2^{1+1/d_n}K s^{(1-\xi)/3}<1$, which is equivalent to $1/(2Kv)>(2/s)^{1/d_n}$. This ends the proof of (4a).

From \eqref{range-s0}, we know that
\begin{equation*}
\begin{split}
1\geq &~2s^{(1-\xi)(1+1/d_n-2\xi/3)}>2^{\frac{1}{d_1}}s^{(1-\xi)(1+1/d_n-2\xi/3)}\\
    = &~ 2^{\frac{1}{d_1}} s^{1-\frac{1}{d_1}}/s^{(\frac{1}{d_2}+\cdots+\frac{1}{d_{n-1}})+\frac{1}{d_n}(\frac{1}{d_1}+\cdots+\frac{1}{d_{n}})+\frac{2\xi(1-\xi)}{3}}\\
  > &~ 2^{\frac{1}{d_1}} s^{1-\frac{1}{d_1}}/s^{(\frac{1}{d_2}+\cdots+\frac{1}{d_{n-1}})+\frac{1}{d_n}(\frac{1}{d_1}+\cdots+\frac{1}{d_{n}})+\frac{1-\xi}{3}(\frac{1}{d_1}+2(\frac{1}{d_2}+\cdots+\frac{1}{d_{n-1}})+\frac{1}{d_n})}\\
  = &~ s^{1-\frac{1}{d_1}} (2/v)^{\frac{1}{d_1}}/|a_1|.
\end{split}
\end{equation*}
This means that $(s/|a_1|)^{d_1}<sv/2=u^{1/2}s^{(1+1/d_n)/2}/2<u^{1/2}/2$. So (4b) holds.

The proof of (4c) is similar to (4b). We just need to note that
\begin{equation*}
1\geq 2s^{(1-\xi)(1+1/d_n-2\xi/3)}>2^{\frac{1}{d_1}(1+\frac{1}{d_n})}s^{(1-\xi)(1+1/d_n-2\xi/3)}
 >  (s/|a_1|)(2v/u)^{\frac{1}{d_1}}(2/s)^{\frac{1}{d_1 d_n}}.
\end{equation*}
This means that $(|a_1|/s)^{d_1}u/(2v)>(2/s)^{1/d_n}$.
\end{proof}

In the following, we use $f$ to denote $f_{p,d_1,\cdots,d_n}$ for simplicity. Note that $0$ and $\infty$ are critical points of $f$ with multiplicity $d_1$ and $d_n$ respectively, and the degree of $f$ is $\sum_{i=1}^{n}d_i$. Denoting $D_i=d_i+d_{i+1}$, we have $5\leq D_i\leq 2K$, where $1\leq i\leq n-1$. Besides $0$ and $\infty$, the rest of the $\sum_{i=1}^{n-1}D_i$ critical points of $f$ are the solutions of
\begin{equation}\label{solu-crit}
(-1)^p \,z\,\frac{f'(z)}{f(z)}=\sum_{i=1}^{n-1}\frac{(-1)^{n-i}D_i z^{D_i}}{z^{D_i}-a_i^{D_i}}+(-1)^n d_1=0.
\end{equation}

For $1\leq i\leq n-1$, let $\widetilde{CP}_i:=\{\widetilde{w}_{i,j}=r_i a_i \exp(\pi \textup{i}\frac{2j-1}{D_i}):1\leq j\leq D_i\}$ be the collection of $D_i$ points lying on the circle $\mathbb{T}_{r_i|a_i|}$ uniformly, where $r_i=\sqrt[D_i]{d_{i}/d_{i+1}}$.
The following lemma shows that the $\sum_{i=1}^{n-1}D_i$ \textit{free} critical points of $f$ are very `close' to $\bigcup_{i=1}^{n-1} \widetilde{CP}_i$.

\begin{lema}\label{crit-close}
For every $\widetilde{w}_{i,j}\in\widetilde{CP}_i$, where $1\leq i\leq n-1$ and $1\leq j\leq D_i$, there exists $w_{i,j}$, which is a solution of \textup{(\ref{solu-crit})}, such that $|w_{i,j}-\widetilde{w}_{i,j}|<u^{\frac{2}{K}}|a_i|$. Moreover, $w_{i_1,j_1}= w_{i_2,j_2}$ if and only if $(i_1,j_1)=(i_2,j_2)$.
\end{lema}
\begin{proof}
Note that the right side of equation \eqref{solu-crit} is equivalent to
\begin{equation}\label{solu-crit-111}
(-1)^{n-i}\left(\frac{D_i z^{D_i}}{z^{D_i}-a_i^{D_i}}-d_{i}\right)+G_i(z)=0,
\end{equation}
where
\begin{equation}\label{G_n}
G_{i}(z)=\sum_{1\leq j\leq n-1,\,j\neq i}\frac{(-1)^{n-j}D_j z^{D_j}}{z^{D_j}-a_j^{D_j}}+(-1)^n d_1+(-1)^{n-i}d_{i}.
\end{equation}
After multiplying both sides of \eqref{solu-crit-111} by $(z^{D_i}-a_i^{D_i})/d_{i+1}$, where $1\leq i\leq n-1$, we have
\begin{equation}\label{solu-crit-3}
(-1)^{n-i}(z^{D_i}+d_{i}a_i^{D_i}/d_{i+1})+(z^{D_i}-a_i^{D_i})\,G_{i}(z)/d_{i+1}=0.
\end{equation}

Let $\Omega_{i}=\{z:|z^{D_i}+d_{i}a_i^{D_i}/d_{i+1}|\leq \varepsilon\,|a_i|^{D_i}\}$, where $\varepsilon=u^{\frac{2}{K}}$ and $1\leq i\leq n-1$. For every $z\in\Omega_{i}$, since $\varepsilon\leq K^{-4}$ by Lemma \ref{esti-a1}(1), we have
\begin{equation}\label{estim-0}
K^{-1}< d_{i}/d_{i+1}-\varepsilon\leq |z/a_i|^{D_i}\leq d_{i}/d_{i+1}+\varepsilon< K-1<K.
\end{equation}
This means that
\begin{equation}\label{estim-00}
K^{-1}< |a_i/z|^{D_i}< K~~\text{and therefore}~~K^{-1}< |a_i/z|^{5}< K.
\end{equation}

If $1\leq j<i$ and $z\in\Omega_{i}$, we have
\begin{equation}\label{less-than-1}
|{a_j}/{z}|^{D_i}\leq|{a_i}/{z}|^{D_i}|{a_{i-1}}/{a_i}|^{D_i}< K u^{1+d_{i+1}/d_{i}}<1.
\end{equation}
Therefore, $|{a_j}/{z}|<1$. By the similar argument, it can be shown that $|z/a_j|<1$ if $i<j\leq n-1$ and $z\in\Omega_{i}$.
If $1\leq j<i$, by Lemma \ref{esti-a1}(1) and (2) and \eqref{estim-00}, we have
\begin{equation}\label{estim-1}
|{a_j}/{z}|^{D_j}\leq|{a_i}/{z}|^{5}|{a_j}/{a_i}|^{5}< K\,\varepsilon^{5(i-j)/2}\leq K^{-9}.
\end{equation}
Similarly, if $i<j\leq n-1$, we have
\begin{equation}\label{estim-2}
|{z}/{a_j}|^{D_j}\leq|{z}/{a_i}|^{5}|{a_i}/{a_j}|^{5}< K\,\varepsilon^{5(j-i)/2}\leq K^{-9}.
\end{equation}
By definition, we have
\begin{equation}\label{estim-3}
\sum_{1\leq j<i}(-1)^{n-j}D_j+(-1)^n d_1+(-1)^{n-i}d_{i}=0.
\end{equation}
From \eqref{G_n}, (\ref{estim-1}), (\ref{estim-2}) and \eqref{estim-3}, we have
\begin{equation*}
\begin{split}
|G_{i}(z)|
 =    &~ \left|\sum_{1\leq j<i}\frac{(-1)^{n-j}D_j}{1-(a_j/z)^{D_j}}+
               \sum_{i< j\leq n-1}\frac{(-1)^{n-j-1}D_j(z/a_j)^{D_j}}{1-(z/a_j)^{D_j}}+(-1)^n d_1+(-1)^{n-i}d_{i}\right|\\
 \leq &~ 2\,K\,\left|\sum_{1\leq j<i}\frac{(-1)^{n-j}(a_j/z)^{D_j}}{1-(a_j/z)^{D_j}}+
               \sum_{i< j\leq n-1}\frac{(-1)^{n-j-1}(z/a_j)^{D_j}}{1-(z/a_j)^{D_j}}\right|\\
 < &~\frac{4 K^2}{1-K^{-9}}\,\sum_{k=1}^{n-1}\varepsilon^{5k/2}<
     \frac{4 K^2}{1-K^{-9}}\,\frac{\varepsilon^{5/2}}{1-\varepsilon^{5/2}}<5\,K^2\,\varepsilon^{5/2}
\end{split}
\end{equation*}
since $\varepsilon^{5/2}\leq K^{-10}$. This means that if $z\in\Omega_{i}$, we have
\begin{equation}
|z^{D_i}-a_i^{D_i}|\cdot|\,G_{i}(z)|/d_{i+1}< 3\,K^3\,\varepsilon^{5/2}|a_i|^{D_i} < \varepsilon|a_i|^{D_i}
\end{equation}
by \eqref{estim-0} and Lemma \ref{esti-a1}(1).

From (\ref{solu-crit-3}) and by Rouch\'{e}'s Theorem, there exists a solution $w_{i,j}$ of \eqref{solu-crit} such that $w_{i,j}\in\Omega_i$ for every $1\leq j\leq D_i$. In particular, $|w_{i,j}-\widetilde{w}_{i,j}|<\varepsilon|a_i|$ by the second statement of Lemma \ref{very-useful-est}(2). Note that for $1\leq i\leq n-2$, we have
\begin{equation}\label{differ-1}
|a_{i+1}|-|a_i|-2\varepsilon|a_i|-2\varepsilon|a_{i+1}|>|a_{i+1}|(1-2\varepsilon-(1+2\varepsilon)K^{-2})>0.
\end{equation}
By Lemma \ref{esti-a1}(1) and $r_i=\sqrt[D_i]{d_{i}/d_{i+1}}\leq (K/2)^{1/5}$, we have,
\begin{equation}\label{differ-2}
\frac{r_i|a_i|\sin(\pi/D_i)}{\varepsilon|a_i|}\geq K^4(\frac{2}{K})^{1/5}\cdot\frac{2}{\pi}\cdot\frac{\pi}{2K}>K^2>1.
\end{equation}
This means that $w_{i_1,j_1}= w_{i_2,j_2}$ if and only if $(i_1,j_1)=(i_2,j_2)$. The proof is complete.
\end{proof}

For $1\leq i\leq n-1$, let $CP_i:=\{w_{i,j}: 1\leq j\leq D_i\}$ be the collection of $D_i$ free critical points of $f$ which lie close to the circle $\mathbb{T}_{r_i|a_i|}$ and denote $CV_i=f(CP_i)$.

\begin{lema}\label{nice-condition}
For every $1\leq i\leq n-1$, there exists an annular neighborhood $A_i$ containing $CP_i\cup\mathbb{T}_{r_i|a_i|}\cup\mathbb{T}_{|a_i|}$, such that

\textup{(1)} If $p=1$, then $f(\overline{A}_i)\subset\mathbb{D}_s$ for odd $n-i$ and $f(\overline{A}_i)\subset\overline{\mathbb{C}}\setminus\overline{\mathbb{D}}_{K}$ for even $n-i$.
In particular, the set of critical values of $f$ satisfies $\bigcup_{i=1}^{n-1}CV_i\subset\mathbb{D}_s \cup \overline{\mathbb{C}}\setminus\overline{\mathbb{D}}_{K}$. The disks $\overline{\mathbb{D}}_s$ and $\overline{\mathbb{C}}\setminus\mathbb{D}_{K}$ lie in the Fatou set of $f$ and $f^{-1}(\overline{\mathbb{A}}_{s,K})\subset \mathbb{A}_{s,K}$.

$(2)$ If $p=0$, then $f(\overline{A}_i)\subset\mathbb{D}_s$ for even $n-i$ and $f(\overline{A}_i)\subset\overline{\mathbb{C}}\setminus\overline{\mathbb{D}}_{M}$ for odd $n-i$, where $M=(2/s)^{1/d_n}$.
In particular, the set of critical values of $f$ satisfies $\bigcup_{i=1}^{n-1}CV_i\subset\mathbb{D}_s \cup \overline{\mathbb{C}}\setminus\overline{\mathbb{D}}_M$. The disks $\overline{\mathbb{D}}_s$ and $\overline{\mathbb{C}}\setminus\mathbb{D}_M$ lie in the Fatou set of $f$ and $f^{-1}(\overline{\mathbb{A}}_{s,M})\subset \mathbb{A}_{s,M}$.
\end{lema}

\begin{proof}
Let $\varepsilon=u^{\frac{2}{K}}\leq K^{-4}$ be the number that appeared in Lemma \ref{crit-close}. For every $1\leq i\leq n-1$, define the annulus
\begin{equation}
A_i=\{z:(\min\{r_i,1\}-2\varepsilon)|a_i|<|z|<(\max\{r_i,1\}+2\varepsilon)|a_i|\}
\end{equation}
where $r_i=\sqrt[D_i]{d_{i}/d_{i+1}}$. Obviously, $A_i\supset CP_i\cup\mathbb{T}_{r_i|a_i|}\cup\mathbb{T}_{|a_i|}$.
By the definition, we have
\begin{equation}
(2/K)^\frac{1}{D_i}\leq \min\{r_i,1\}\leq \max\{r_i,1\}\leq (K/2)^\frac{1}{D_i}.
\end{equation}

If $z\in\overline{A}_i$, we have
\begin{equation}\label{esti-below}
|a_i/z|\leq\frac{1}{(2/K)^\frac{1}{D_i}-2\varepsilon}\leq \frac{(K/2)^\frac{1}{D_i}}{1-2K^{-4}(K/2)^{1/5}}<(K/2)^\frac{1}{D_i}(1+4/K^{19/5}).
\end{equation}
and
\begin{equation}\label{esti-above}
|z/a_i|\leq(K/2)^\frac{1}{D_i}+2\varepsilon\leq (K/2)^\frac{1}{D_i}+2/K^4<(K/2)^\frac{1}{D_i}(1+1/K^3).
\end{equation}
This means that
\begin{equation}\label{esti-below-new}
|a_i/z|^5<(K/2)^\frac{5}{D_i}(1+4/K^{19/5})^{5}<(K/2)\,e^{20/K^{19/5}}< (K/2)\,e^{20/3^{19/5}}<7K/10.
\end{equation}
and also,
\begin{equation}\label{esti-above-new}
|z/a_i|^5<(K/2)^\frac{5}{D_i}(1+1/K^3)^{5}<(K/2)\,e^{5/K^3}< (K/2)\,e^{5/27}<7K/10.
\end{equation}
Moreover, similar to the argument of \eqref{esti-below-new} and \eqref{esti-above-new}, we have
\begin{equation}\label{esti-other}
|a_i/z|^{d_{i}}+|z/a_i|^{d_{i+1}}<7K/5.
\end{equation}

Recall that $|a_i/a_{i+1}|^{d_{i+1}}=u$ for every $1\leq i\leq n-2$ and $|a_{n-1}|^{d_n}=v$. Let $1\leq i_1\leq i_2\leq n-1$ and $p\in\{0,1\}$, we have
\begin{equation}\label{sequence}
\begin{split}
\prod_{j=i_1}^{i_2}|a_j|^{(-1)^{n-j-p}D_j}=
   &~ |a_{i_1}|^{(-1)^{n-i_1-p}d_{i_1}}\,|a_{i_2}|^{(-1)^{n-i_2-p}d_{i_2+1}}
      \,\prod_{j=i_1}^{i_2-1}\left|\frac{a_{j}}{a_{j+1}}\right|^{(-1)^{n-j-p}d_{j+1}} \\
 = &~ |a_{i_1}|^{(-1)^{n-i_1-p}d_{i_1}}\,|a_{i_2}|^{(-1)^{n-i_2-p}d_{i_2+1}}\,u^{\frac{(-1)^{n-i_1-p}-(-1)^{n-i_2-p}}{2}}\\
 = &~
\left\{                         %方程组的左边包括大括号\{
\begin{array}{ll}               %设定列阵的格式：{ll}是2个L，表示2列的对齐方式为Left对齐
(|a_1|^{d_1}u/v)^{(-1)^p}  &~~\text{if}~~i_1=1~~\text{and}~~i_2=n-1~~\text{is even} \\
(|a_1|^{-d_1}/v)^{(-1)^p}  &~~\text{if}~~i_1=1~~\text{and}~~i_2=n-1~~\text{is odd}.    %$同上
\end{array}                     %方程列阵的结束
\right.                         %方程组的右边无符号，利用“.“来标示
\end{split}
\end{equation}

By \eqref{family} and the second equation of \eqref{sequence}, we have
\begin{equation}\label{abs-f-n}
\begin{split}
    &~|f(z)| \\
=  &~ |z^{D_i}-a_i^{D_i}|^{(-1)^{n-i-p}}\,|z|^{(-1)^{n-p} d_1}\,\prod_{j=1}^{i-1}|z|^{(-1)^{n-j-p}D_j}
         \,\prod_{j=i+1}^{n-1}|a_j|^{(-1)^{n-j-p}D_j}\cdot Q_i(z) \\
 = &~ |1-(z/a_i)^{D_i}|^{(-1)^{n-i-p}}\,|{z}/{a_i}|^{(-1)^{n-i-p+1}d_{i}}\,|a_{n-1}|^{(-1)^{1-p}d_n}
      \,u^{\frac{(-1)^{n-i-p}-(-1)^{1-p}}{2}}\cdot Q_i(z) \\
 = &~ v^{(-1)^{1-p}}\,u^{\frac{(-1)^{n-i-p}-(-1)^{1-p}}{2}}\,|(a_i/z)^{d_{i}}-(z/a_i)^{d_{i+1}}|^{(-1)^{n-i-p}}\cdot Q_i(z)\\
  &~
\left\{                         %方程组的左边包括大括号\{
\begin{array}{ll}               %设定列阵的格式：{ll}是2个L，表示2列的对齐方式为Left对齐
\leq v^{(-1)^{1-p}}\,u^{\frac{1-(-1)^{1-p}}{2}}\,(|a_i/z|^{d_{i}}+|z/a_i|^{d_{i+1}})\,Q_i(z)  &~~\text{if}~~n-i-p~~\text{is even} \\
\geq v^{(-1)^{1-p}}\,u^{\frac{-1-(-1)^{1-p}}{2}}\,(|a_i/z|^{d_{i}}+|z/a_i|^{d_{i+1}})^{-1}\,Q_i(z)  &~~\text{if}~~n-i-p~~\text{is odd},    %$ 同上
\end{array}                     %方程列阵的结束
\right.                         %方程组的右边无符号，利用“.“来标示
\end{split}
\end{equation}
where
\begin{equation}\label{Q-i}
Q_i(z)=\prod_{j=1}^{i-1}\left|1-({a_j}/{z})^{D_j}\right|^{(-1)^{n-j-p}}
       \prod_{j=i+1}^{n-1}\left|1-({z}/{a_j})^{D_j}\right|^{(-1)^{n-j-p}}.
\end{equation}

For $1\leq i\leq n-1$, consider $z\in \overline{A}_i$. If $1\leq j<i$, by \eqref{esti-below-new}, we have
\begin{equation}\label{estim-1-2}
|{a_j}/{z}|^{D_j}\leq|{a_i}/{z}|^{5}|{a_j}/{a_i}|^{5}< 7K\,\varepsilon^{5(i-j)/2}/10<K^{-9}.
\end{equation}
If $i<j\leq n-1$, then
\begin{equation}\label{estim-2-2}
|{z}/{a_j}|^{D_j}\leq|{z}/{a_i}|^{5}|{a_i}/{a_j}|^{5}< 7K\,\varepsilon^{5(i-j)/2}/10<K^{-9}.
\end{equation}
by \eqref{esti-above-new}. Since $e^x<1+2x$ if $0<x\leq 1$ and $\varepsilon\leq K^{-4}$, by \eqref{Q-i}--\eqref{estim-2-2}, we have
\begin{equation}\label{Q-i-esti-1}
Q_i(z)<  \prod_{k=1}^{\infty}\left(1+7K\,\varepsilon^{5k/2}/5\right)^2
        \leq \exp\left(\frac{14\,K\,\varepsilon^{5/2}/5}{1-\varepsilon^{5/2}}\right)<1+K^{-5}<1.01.
\end{equation}
and
\begin{equation}\label{Q-i-esti-2}
Q_i(z)>  \prod_{k=1}^{\infty}\left(1+7K\,\varepsilon^{5k/2}/5\right)^{-2}
        >   1/1.01 > 0.99.
\end{equation}

For $p=1$, by Lemma \ref{esti-a1}(2) and (3a), for every $1\leq i\leq n-1$, if $|z|\leq s$, we have
\begin{equation}\label{last-estima-0}
|z^{D_i}/a_i^{D_i}| \leq   |s/a_1|^{D_i}|a_1/a_i|^{D_i}\leq (sK^{-3}/2)^{\frac{5}{K}}u^{\frac{5(i-1)}{K}}.
\end{equation}
If we notice Lemma \ref{esti-a1}(1), then
\begin{equation}\label{last-estima-00}
\sum_{i=1}^{n-1}|z^{D_i}/a_i^{D_i}| \leq \frac{(sK^{-3}/2)^{\frac{5}{K}}}{1-u^{\frac{5}{K}}}
\leq\frac{K^{\frac{10}{K}-10}}{1-K^{-10}}<1/200.
\end{equation}

For $p=0$, by Lemma \ref{esti-a1}(2) and (4b), for every $1\leq i\leq n-1$, if $|z|\leq s$, we have
\begin{equation}\label{last-estima-0-lp}
|z^{D_i}/a_i^{D_i}| \leq   |s/a_1|^{D_i}|a_1/a_i|^{D_i}\leq (u^{1/2}/2)^{\frac{5}{K}}u^{\frac{5(i-1)}{K}}.
\end{equation}
By Lemma \ref{esti-a1}(1), then
\begin{equation}\label{last-estima-00-lp}
\sum_{i=1}^{n-1}|z^{D_i}/a_i^{D_i}| \leq \frac{(u^{1/2}/2)^{\frac{5}{K}}}{1-u^{\frac{5}{K}}}
\leq\frac{K^{-5}}{1-K^{-10}}<1/200.
\end{equation}

Since $(1+2|a|)^{-1}\leq |1+a|^{\pm 1}\leq 1+2|a|$ if $0\leq |a|\leq 1/2$, by \eqref{last-estima-00} and \eqref{last-estima-00-lp}, we know that
\begin{equation}\label{last-estima-1}
\prod_{i=1}^{n-1}\left|1-{z^{D_i}}/{a_i^{D_i}}\right|^{(-1)^{n-i-p}}
     \leq \prod_{i=1}^{n-1}\left(1+2|z/a_i|^{D_i}\right)<e^{1/100}<K.
\end{equation}
Therefore,
\begin{equation}\label{last-estima-2}
\prod_{i=1}^{n-1}\left|1-{z^{D_i}}/{a_i^{D_i}}\right|^{(-1)^{n-i-p}}
      \geq \prod_{i=1}^{n-1}\left(1+2|z/a_i|^{D_i}\right)^{-1}>e^{-1/100}>1/K.
\end{equation}

(1) We first consider the case $p=1$.
If $n-i$ is odd, by \eqref{esti-other}, \eqref{abs-f-n} and \eqref{Q-i-esti-1}, if $z\in \overline{A}_i$ we have
\begin{equation}\label{bound-f-n-0}
|f(z)|\leq v\cdot (7K/5)\cdot 1.01<2Kv<s.
\end{equation}
If $n-i$ is even, by \eqref{esti-other}, \eqref{abs-f-n} and \eqref{Q-i-esti-2}, for $z\in \overline{A}_i$ we have
\begin{equation}\label{bound-f-n-1}
|f(z)|\geq (v/u)\cdot (7K/5)^{-1}\cdot 0.99>v/(2Ku)>K.
\end{equation}

If $n$ is odd, by Lemma \ref{esti-a1}(3a), \eqref{sequence} and \eqref{last-estima-1}, for every $z$ such that $|z|\leq s$, we have
\begin{equation*}
|f(z)| =|z|^{d_1} \prod_{i=1}^{n-1}|a_i|^{D_i(-1)^{n-i-1}} \prod_{i=1}^{n-1}\left|1-\frac{z^{D_i}}{a_i^{D_i}}\right|^{(-1)^{n-i-1}}
      < |s/a_1|^{d_1}vu^{-1}\cdot1.02<s.
\end{equation*}
It follows that $f(\overline{\mathbb{D}}_{s})\subset\mathbb{D}_{s}$ for odd $n$. If $n$ is even and $|z|\leq s$, by Lemma \ref{esti-a1}(3b), \eqref{sequence} and \eqref{last-estima-2}, we have
\begin{equation*}
|f(z)|=|a_1/z|^{d_1}v\,\prod_{i=1}^{n-1}\left|1-\frac{z^{D_i}}{a_i^{D_i}}\right|^{(-1)^{n-i-1}}
 >  |a_1/s|^{d_1}v/1.02>K.
\end{equation*}
Therefore $f(\overline{\mathbb{D}}_{s})\subset\overline{\mathbb{C}}\setminus\overline{\mathbb{D}}_{K}$ for even $n$.

Note that $f$ is very `close' to $z\mapsto z^{d_n}$ in the outside of $\mathbb{D}_{K}$ since $|a_i|^{D_i}$ is extremely small, where $1\leq i\leq n-1$.
This means that $f$ may exhibit some dynamics of $z\mapsto z^{d_n}$ if $|z|\geq K$. More specifically, by arguments completely similar to those for \eqref{last-estima-1}--\eqref{last-estima-2}, if $|z|\geq K$, then
\begin{equation}\label{bound-lower-out-disk}
|f(z)|\geq |z|^{d_n} \prod_{i=1}^{n-1}\left(1+2\frac{|a_i|^{D_i}}{|z|^{D_i}}\right)^{-1}>K.
\end{equation}
This means that $f(\overline{\mathbb{C}}\setminus\mathbb{D}_{K})\subset\overline{\mathbb{C}}\setminus\overline{\mathbb{D}}_{K}$.
Then we have $f^{-1}(\overline{\mathbb{A}}_{s,K})\subset \mathbb{A}_{s,K}$ for every $n\geq 2$ (see Figure \ref{Fig_cantor-gene}).

(2) Now we consider the case $p=0$.
If $n-i$ is even, by \eqref{esti-other}, \eqref{abs-f-n}, \eqref{Q-i-esti-1} and Lemma \ref{esti-a1}(4a), if $z\in \overline{A}_i$ we have
\begin{equation}\label{bound-f-n-0-lp}
|f(z)|\leq v^{-1}u\cdot (7K/5)\cdot 1.01<2Ku/v<s.
\end{equation}
If $n-i$ is odd, by \eqref{esti-other}, \eqref{abs-f-n}, \eqref{Q-i-esti-2} and Lemma \ref{esti-a1}(4a), for $z\in \overline{A}_i$ we have
\begin{equation}\label{bound-f-n-1-lp}
|f(z)|\geq v^{-1}\cdot (7K/5)^{-1}\cdot 0.99>1/(2Kv)>M,
\end{equation}
where $M=(2/s)^{1/d_n}$.

If $n$ is even, by Lemma \ref{esti-a1}(4b), \eqref{sequence} and \eqref{last-estima-1}, for each $z$ such that $|z|\leq s$, we have
\begin{equation*}
|f(z)| =|z|^{d_1} \prod_{i=1}^{n-1}|a_i|^{D_i(-1)^{n-i}} \prod_{i=1}^{n-1}\left|1-\frac{z^{D_i}}{a_i^{D_i}}\right|^{(-1)^{n-i}}
      < |s/a_1|^{d_1}v^{-1}\cdot e^{1/100}<s.
\end{equation*}
It follows that $f(\overline{\mathbb{D}}_{s})\subset\mathbb{D}_{s}$ for even $n$. If $n$ is odd and $|z|\leq s$, by Lemma \ref{esti-a1}(4c), \eqref{sequence} and \eqref{last-estima-2}, we have
\begin{equation*}
|f(z)|=|a_1/z|^{d_1}uv^{-1}\,\prod_{i=1}^{n-1}\left|1-\frac{z^{D_i}}{a_i^{D_i}}\right|^{(-1)^{n-i}}\geq |a_1/s|^{d_1}uv^{-1}\cdot e^{-1/100}
 >  M.
\end{equation*}
Therefore $f(\overline{\mathbb{D}}_{s})\subset\overline{\mathbb{C}}\setminus\overline{\mathbb{D}}_{M}$ for odd $n$.

If $|z|\geq M$, then
\begin{equation}\label{bound-lower-out-disk-lp}
|f(z)| =|z|^{-d_n} \prod_{i=1}^{n-1}\left|1-\frac{a_i^{D_i}}{z^{D_i}}\right|^{(-1)^{n-i}}
      \leq M^{-d_n} \prod_{i=1}^{n-1}\left(1+\frac{2 |a_i|^{D_i}}{|z|^{D_i}}\right)<2M^{-d_n}=s.
\end{equation}
This means that $f(\overline{\mathbb{C}}\setminus\mathbb{D}_M)\subset\mathbb{D}_{s}$.
Then we have $f^{-1}(\overline{\mathbb{A}}_{s,M})\subset \mathbb{A}_{s,M}$ for every $n\geq 2$.
\end{proof}

\begin{figure}[!htpb]
  \setlength{\unitlength}{1mm}
  \centering
  \includegraphics[width=130mm]{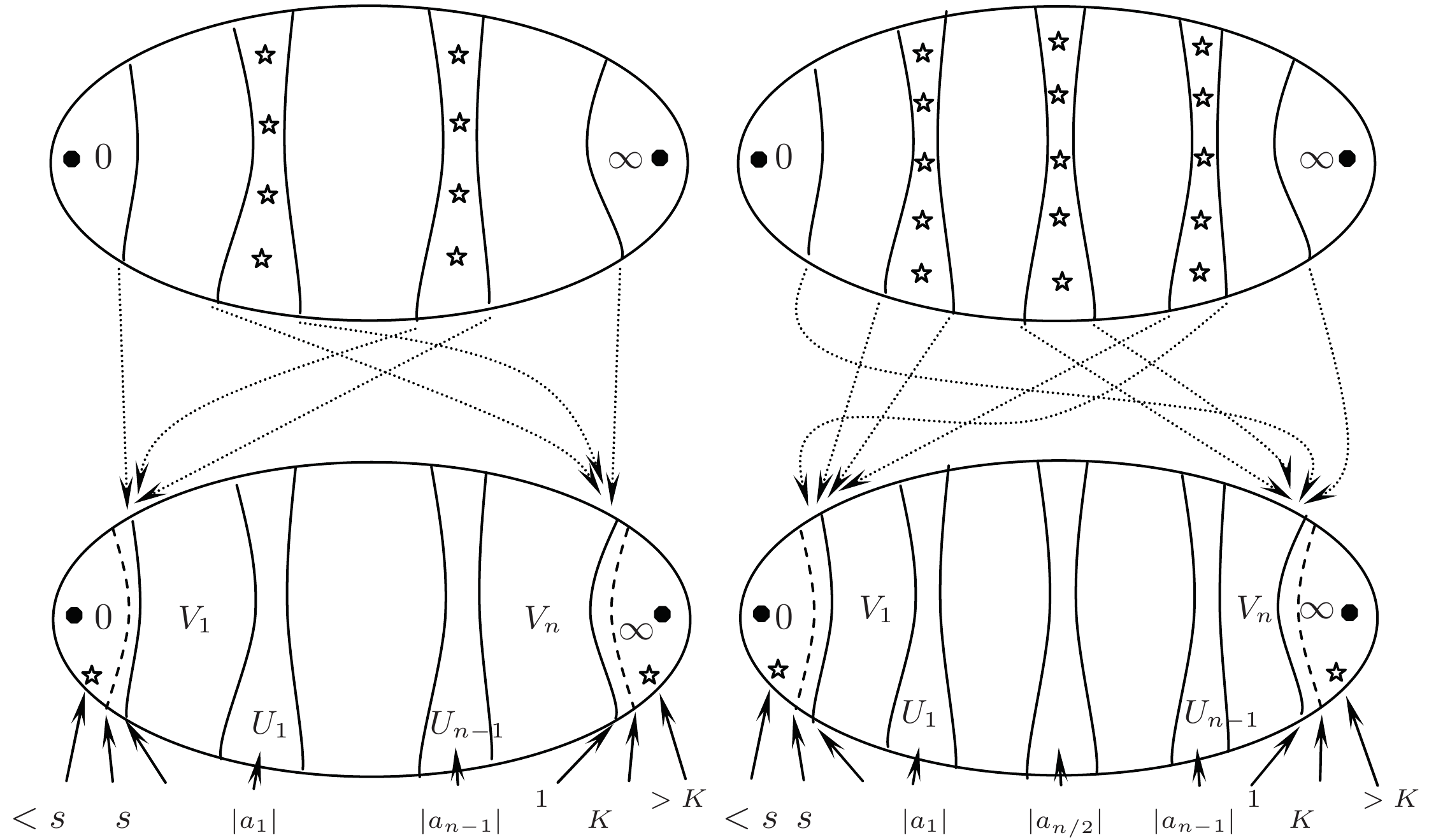}
  \caption{Sketch illustrating of the mapping relation of $f_{1,d_1,\cdots,d_n}$, where $n$ is odd and even respectively (from left to right). The small stars denote the critical points and critical values, and the numbers shown at the bottom of the Figures denote the approximate coordinates.}
  \label{Fig_cantor-gene}
\end{figure}

\begin{thm}\label{parameter-restate}
If $|a_{n-1}|=(s_1 K^{-2})^{1/d_n}$ and $|a_i|=(s_1 K^{-5})^{1/d_{i+1}}|a_{i+1}|$ for $1\leq i\leq n-2$, where $s_1>0$ is small enough,
then the Julia set of $f_{1,d_1,\cdots,d_n}$ is a Cantor set of circles. If $|a_{n-1}|=(s_0^{1/d_n+(1-\xi)/3})^{1/d_n}$ and $|a_i|=(s_0^{1+1/d_n+2(1-\xi)/3})^{1/d_{i+1}}|a_{i+1}|$ for $1\leq i\leq n-2$, where $s_0>0$ is small enough,
then the Julia set of $f_{0,d_1,\cdots,d_n}$ is a Cantor set of circles.
\end{thm}

\begin{proof}
We only focus on the case $p=1$ since the similar proof can be used to the case $p=0$ by using Lemma \ref{nice-condition}(2).
We also use $f$ to denote $f_{1,d_1,\cdots,d_n}$ for simplicity.
Let $U_i$ be the component of $f^{-1}(D)$ containing $a_i$, where $D=\mathbb{D}_s$ if $n-i$ is odd and $D=\overline{\mathbb{C}}\setminus\overline{\mathbb{D}}_{K}$ if $n-i$ is even. By Lemma \ref{nice-condition}(1), it follows that the set of critical points $CP_i\subset U_i$ and $U_i$ is a connected domain containing the annulus $A_i$. Moreover, $U_i\cap U_{i+1}=\emptyset$ since $f(U_i)\cap f(U_{i+1})=\emptyset$ by Lemma \ref{nice-condition}(1), where $1\leq i<n-2$. This means that $U_i\cap U_j=\emptyset$ for different $i,j$. Suppose that $U_i$ has $m_i$ boundary components. Since there are exactly $D_i$ critical points in $U_i$ and $f:U_i\rightarrow D$ is a branched covering with degree $D_i$, then the Riemann-Hurwitz formula tells us $\chi_{U_i}=2-m_i=D_i\chi_{D}-D_i=0$, where $\chi$ denotes the Euler characteristic. This means that $m_i=2$ and therefore $U_i$ is an annulus surrounding the origin for every $1\leq i\leq n-1$.

For $1\leq i\leq n-2$, Let $V_{i+1}$ be the annular domain between $U_i$ and $U_{i+1}$. It is easy to see $f:V_{i+1}\rightarrow \mathbb{A}_{s,{K}}$ is a covering map with degree $d_{i+1}$. Note that every component of $f^{-1}(\mathbb{A}_{s,{K}})$ is an annulus since $\mathbb{A}_{s,{K}}$ is double connected and contains no critical values. It follows that there exist two annuli $V_1$ and $V_n$, which lie between $0$ and $U_1$, $U_{n-1}$ and $\infty$ respectively, such that $f:V_1,V_n\rightarrow \mathbb{A}_{s,{K}}$ are covering maps with degree $d_1$ and $d_n$ respectively. In fact, the restriction of $f$ on $\partial U_1$ and $\partial U_{n-1}$ has degree $d_1$ and $d_n$ respectively and there are no critical points in $V_1$ and $V_n$ (see Figure \ref{Fig_cantor-gene}).

The Julia set of $f$ is $J=\bigcap_{k\geq 0}f^{-k}(\mathbb{A}_{s,{K}})$. By the construction, the components of $J$ are compact sets nested between $0$ and $\infty$ since each inverse branch $f^{-1}:\mathbb{A}_{s,{K}}\rightarrow V_j$ is conformal for every $0\leq j\leq n$. Since the component of $J$ cannot be a point and $f$ is hyperbolic, every component of $J$ is a Jordan curve (actually quasicircle) by Theorem 1.2 in \cite{PT}. The dynamics on the set of Julia components of $f$ is isomorphic to the one-sided shift on $n$ symbols $\Sigma_{n}:=\{0,1,\cdots,n-1\}^{\mathbb{N}}$. In particular, $J$ is homeomorphic to $\Sigma_{n}\times\mathbb{S}^1$, which is a Cantor set of circles as desired. This ends the proof of Theorem \ref{parameter-restate} and hence Theorem \ref{parameter}.
\end{proof}

\begin{rmk}\label{range-unif}
Since $f$ is hyperbolic, the Julia set of $f$ is also a Cantor set of circles if we perturb some $a_i$ gently, where $1\leq i\leq n-1$.
In the first version of our manuscript of this paper, only $d_i=n+1$ for every $1\leq i\leq n$ was considered. In this case, it was shown that for every $n\geq 2$ and $1\leq i\leq n-1$, if $|a_{n-i}|=(\frac{n}{n+1})^{i-1}s^i$ for $0<s\leq 1/10$, then the Julia set of $f_{1,n+1,\cdots,n+1}$ is a Cantor set of circles.
\end{rmk}

\begin{thm}\label{no-topo-equiv}
Suppose that $a_i$ is chosen as in Theorem \ref{parameter} such that the Julia set of $f_{p,d_1,\cdots,d_n}$ is a Cantor set of circles for $n\geq 3$, then $f_{p,d_1,\cdots,d_n}$ is not topologically conjugate to any McMullen maps on their corresponding Julia sets.
\end{thm}

\begin{proof}
Since the dynamics on the set of Julia components of $f_{p,d_1,\cdots,d_n}$ is conjugate to the one-sided shift on $n$ symbols $\Sigma_n:=\{0,1,\cdots,n-1\}^{\mathbb{N}}$ and, in particular, the set of Julia components of $g_\eta$ is isomorphic to the one-sided shift on only two symbols $\Sigma_{2}:=\{0,1\}^{\mathbb{N}}$, this means that $f_{p,d_1,\cdots,d_n}$ cannot be topologically conjugate to $g_\eta$ on their corresponding Julia sets if $n\geq 3$.
\end{proof}

\section{Topological conjugacy between the Cantor circles Julia sets}\label{sec-topo-conj}

In this section, we show that for any given rational map whose Julia set is a Cantor set of circles, there exists a map $f_{p,d_1,\cdots,d_n}$ in \eqref{family} such that these two rational maps are topologically conjugate on their corresponding Julia sets.

\begin{lema}\label{no-crit-on-J}
If $f$ is a rational map whose Julia set is a Cantor set of circles. Then there exist no critical points in $J(f)$.
\end{lema}

\begin{proof}
Suppose there exists a Julia component $J_0$ of $f$ containing a critical point $c_0$ of $f$ with multiplicity $d$. Then $f$ is not one to one in any small neighborhood of $c_0$. It is known $f(J_0)$ is a Julia component containing $f(c_0)$ \cite[Lemma 5.7.2]{Be}. Choose a small topological disk neighborhood $U$ of $f(c_0)$ such that $U\cap f(J_0)$ is a simple curve. The component of $f^{-1}(U)$ containing $c_0$ is mapped onto $U$ in the manner of $d+1$ to one. Note that the component $J'$ of $f^{-1}(U\cap f(J_0))$ containing $c_0$ is connected and contained in $J_0$. However, $J'$ possesses star-like structure and hence is not a simple curve. This contradicts to the assumption that $J_0$ is a Jordan closed curve since $J(f)$ is a Cantor set of circles.
\end{proof}

We say that a compact set $X\subset \overline{\mathbb{C}}$ \textit{separates} $0$ and $\infty$ if $0$ and $\infty$ lie in the two different components of $\overline{\mathbb{C}}\setminus X$ respectively. Let $X$ and $Y$ be two disjoint compact sets that both separate $0$ and $\infty$ respectively. We say $X\prec Y$ if $X$ is contained in the component of $\overline{\mathbb{C}}\setminus Y$ which contains $0$.  Let $A$ be an annulus whose closure separates $0$ and $\infty$, we use $\partial_-A$ and $\partial_+A$ to denote the two components of the boundary of $A$ such that $\partial_-A\prec \partial_+A$.

\begin{thm}\label{this-is-all-resta}
Let $f$ be a rational map whose Julia set is a Cantor set of circles. Then there exist $p\in\{0,1\}$, positive integers $n\geq 2$ and $d_1,\cdots,d_n$ satisfying $\sum_{i=1}^{n}(1/d_i)<1$ such that $f$ is topologically conjugate to $f_{p,d_1,\cdots,d_n}$ on their corresponding Julia sets.
\end{thm}

\begin{proof}
Let $J(f)$ be the Julia set of $f$ which is a Cantor set of circles, then every periodic Fatou component of $f$ must be attracting or parabolic by Lemma \ref{no-crit-on-J}. We only prove the attracting (hyperbolic) case in detail and explain the parabolic case by using the work of Cui \cite{Cui}.

In the following, we suppose that $f$ is hyperbolic. There exist exactly two simply connected Fatou components of $f$ and all other Fatou components are annuli.
Let $\mathcal{D}$ and $\mathcal{A}$ be the collection of simply and doubly connected Fatou components of $f$ respectively. We claim that $f(\mathcal{D})\subset\mathcal{D}$ and there exists an integer $k\geq 1$ such that $f^{\circ k}(A)\in\mathcal{D}$ for every $A\in \mathcal{A}$. The assertion $f(\mathcal{D})\subset\mathcal{D}$ is obvious since the image of a simply connected Fatou component under a rational map is again simply connected. If $f(A_1)=A_2$, where $A_1,A_2\in\mathcal{A}$, then there exists no critical points in $A_1$ by Riemann-Hurwitz's formula. This means that each $A\in\mathcal{A}$ cannot be periodic since the cycle of every periodic attracting Fatou component must contain at least one critical point. On the other hand, by Sullivan's theorem, the Fatou components of a rational map cannot be wandering. This completes the proof of claim.

Up to a Mobius transformation, we can assume that $0$ and $\infty$, respectively, are belong to the two simply connected Fatou components of $f$, which are denoted by $D_0$ and $D_\infty$. Namely, $\mathcal{D}=\{D_0,D_\infty\}$. Since $f(\mathcal{D})\subset\mathcal{D}$, we first suppose that $f(D_0)=D_0$ and $f(D_\infty)=D_\infty$. Let $f^{-1}(D_0)=D_0\cup A_1\cup\cdots\cup A_m$, where $A_1,\cdots, A_m$ are $m$ annuli separating $0$ and $\infty$ such that $A_{i}\prec A_{i+1}$ for every $1\leq i\leq m-1$. It is easy to see $m\geq 1$. Otherwise, $D_0$ is completely invariant, then $J(f)=\partial D_0$ which contradicts to the assumption that $J(f)$ is a Cantor set of circles.

Suppose that $\deg (f|_{D_0}:D_0\rightarrow D_0)=d_1$ and $\deg (f|_{\partial_-A_i}:\partial_-A_i\rightarrow \partial D_0)=d_{2i}$ and $\deg (f|_{\partial_+A_i}:\partial_+A_i\rightarrow \partial D_0)=d_{2i+1}$ for $1\leq i\leq m$. It follows that $\deg(f)=\sum_{j=1}^{2m+1}d_j$. Let $W_1$ be the annular domain between $D_0$ and $A_1$ and $W_i$ be the annular domain between $A_{i-1}$ and $A_i$, where $2\leq i\leq m$. We have $f(W_i)=\overline{\mathbb{C}}\setminus \overline{D}_0$ and $\deg(f|_{W_i}:W_i\rightarrow\overline{\mathbb{C}}\setminus \overline{D}_0)=d_{2i-1}+d_{2i}$. This means that there exists at least one Fatou component $B_i\subsetneq W_i$ such that $f(B_i)=D_\infty$. If there exists $B_i'\neq B_i$ such that $B_i'\subsetneq W_i$ and $f(B_i')=D_\infty$, there must exist one component of $f^{-1}(D_0)$ in $W_i$, which contradicts the assumption that $A_1\cup\cdots\cup A_m$ is the collection of all annular components of $f^{-1}(D_0)$. So there exists exactly one Fatou component $B_i\subsetneq W_i$ such that $f(B_i)=D_\infty$ and $\deg(f|_{B_i}:B_i\rightarrow D_\infty)=d_{2i-1}+d_{2i}$. Similar argument can be used to show that $D_\infty$ is the only component of $f^{-1}(D_\infty)$ lying in the unbounded component of $\overline{\mathbb{C}}\setminus A_m$ which can be mapped onto $D_\infty$. Therefore, $f^{-1}(D_\infty)=B_1\cup\cdots\cup B_m\cup D_\infty$ and $\deg(f|_{D_\infty})=d_{2m+1}$ since $\deg(f)=\sum_{j=1}^{2m+1}d_j$.
Denote $\overline{\mathbb{C}}\setminus ({D_0\cup D_\infty})$ by $E$. The preimage $f^{-1}(E)$ consists of $2m+1$ annuli components $E_1,\cdots,E_{2m+1}$ such that $E_i\prec E_{i+1}$ for $1\leq i\leq 2m$. The map $f:E_i\rightarrow E$ is a unramified covering map with degree $d_i$, where $1\leq i\leq 2m+1$ (see Figure \ref{Fig_find-conj}).

\begin{figure}[!htpb]
  \setlength{\unitlength}{1mm}
  \centering
  \includegraphics[width=120mm]{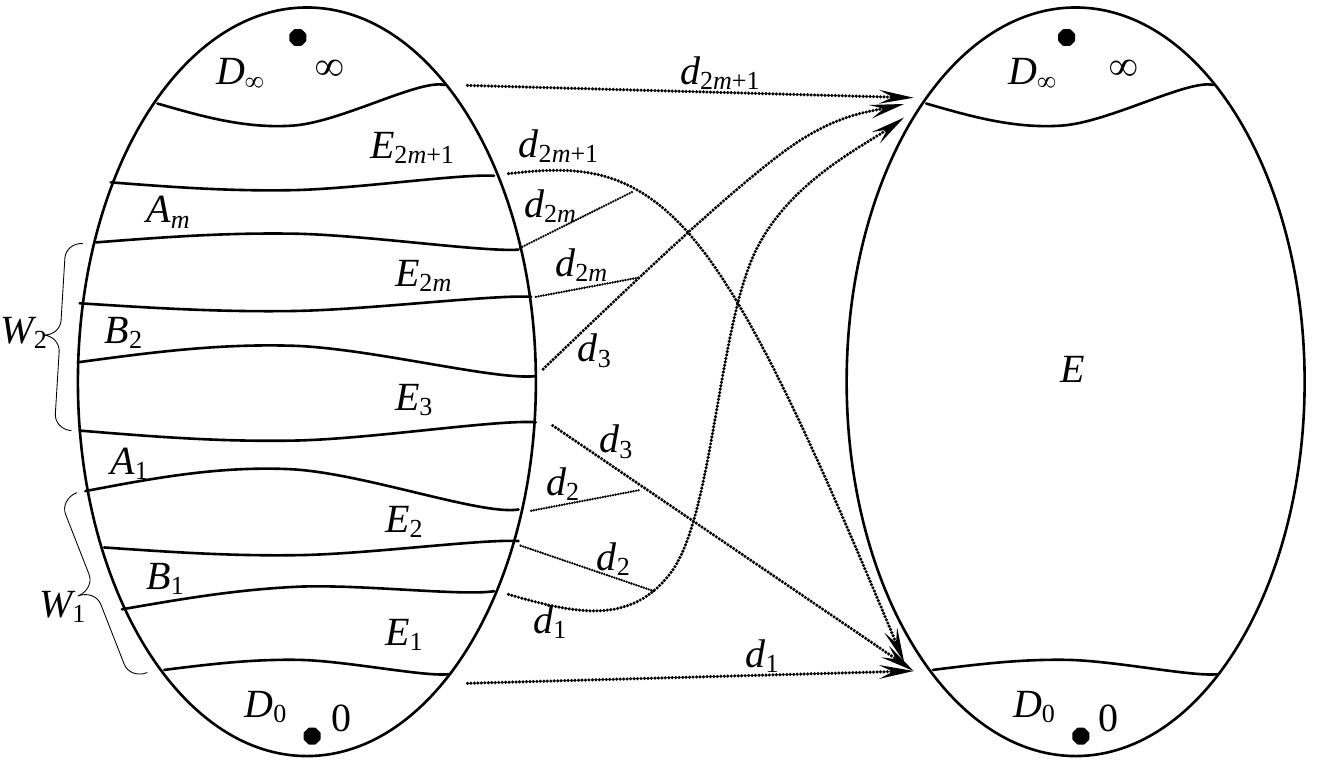}
  \caption{Sketch illustrating of the mapping relation of $f$, where $d_i$, $1\leq i\leq 2m+1$ denote the degrees of the restriction of $f$ on the boundaries of Fatou components.}
  \label{Fig_find-conj}
\end{figure}

Let $n=2m+1$ and $p=1$. The assertion $\sum_{i=1}^{n}{1}/{d_i}<1$ follows from Gr\'{o}tzsch's modulus inequality since each $E_i$ is essentially contained in $E$ and $\text{mod} (E_i)=\text{mod} (E)/d_i$. In the following, we will construct a quasiconformal map $\phi:\overline{\mathbb{C}}\rightarrow\overline{\mathbb{C}}$ which conjugates the dynamics on the Julia set of $f$ to that of $f_{1,d_1,\cdots,d_n}$.

For simplicity, we denote $f_{1,d_1,\cdots,d_n}$ by $F$. Note that $F(0)=0$ and $F(\infty)=\infty$. There exist two simply connected Fatou components $D_0'$ and $D_\infty'$, both are invariant under $F$ such that $0\in D_0'$ and $\infty\in D_\infty'$. From the proof of Theorem \ref{parameter}, we know that  $F^{-1}(D_0')=D_0'\cup A_1'\cup\cdots\cup A_m'$, where $A_1',\cdots, A_m'$ are $m$ annuli separating $0$ and $\infty$ such that $A_{i}'\prec A_{i+1}'$ for every $1\leq i\leq m-1$. Moreover, $\deg (F|_{D_0'}:D_0'\rightarrow D_0')=d_1$ and $\deg (F|_{\partial_-A_i'}:\partial_-A_i'\rightarrow \partial D_0')=d_{2i}$ and $\deg (F|_{\partial_+A_i'}:\partial_+A_i'\rightarrow \partial D_0')=d_{2i+1}$ for $1\leq i\leq m$. Let $W_1'$ be the annular domain between $D_0'$ and $A_1'$ and $W_i'$ be the annular domain between $A_{i-1}'$ and $A_i'$, where $2\leq i\leq m$. There exists exactly one Fatou component $B_i'\subsetneq W_i'$ such that $F(B_i')=D_\infty'$ and $\deg(F|_{B_i'}:B_i'\rightarrow D_\infty')=d_{2i-1}+d_{2i}$. We have $F^{-1}(D_\infty')=B_1'\cup\cdots\cup B_m'\cup D_\infty'$ and $\deg(F|_{D_\infty'})=d_{2m+1}$.
Similarly, let $E':=\overline{\mathbb{C}}\setminus ({D_0'\cup D_\infty'})$. There exist $2m+1$ annular components $E_1',\cdots,E_{2m+1}'$ of $F^{-1}(E')$ such that $E_i'\prec E_{i+1}'$ for $1\leq i\leq 2m$. The map $F:E_i'\rightarrow E'$ is a covering with degree $d_i$, where $1\leq i\leq 2m+1$.

By a quasiconformal surgery, it can be seen that $\partial D_0,\partial D_\infty,\partial D_0',\partial D_\infty'$ and their preimages are all quasicircles and the dilatation is bounded by a fixed constant. There exists a quasiconformal mapping $\phi_0:\overline{\mathbb{C}}\rightarrow\overline{\mathbb{C}}$ such that $\phi_0(D_0)=D_0'$ and $\phi_0(D_\infty)=D_\infty'$ hence $\phi_0(\partial D_0)=\partial D_0'$ and $\phi_0(\partial D_\infty)=\partial D_\infty'$. Moreover, $\phi_0$ can be chosen such that $\phi_0\circ f=F\circ\phi_0$ on $\partial D_0\cup \partial D_\infty$.

Now we construct a lift $\phi_{E_1}:E_1\rightarrow E_1'$ of $\phi_0:E\rightarrow E'$ as follows. For every $z\in E_1\setminus\partial_- E_1$, we choose a simple curve $\gamma:[0,1]\rightarrow E$ such that $\gamma(1)=f(z)$ and $\gamma(0)=w\in\partial_- E$. Since $f:E_1\rightarrow E$ is a covering map, there exists a unique lift $\widetilde{\gamma}:[0,1]\rightarrow E_1$ of $\gamma$ such that $\widetilde{\gamma}(1)=z$ and $\widetilde{w}:=\widetilde{\gamma}(0)\in\partial_- E_1$. Similarly, since $F:E_1'\rightarrow E'$ is a covering map, there exists a unique lift $\alpha:[0,1]\rightarrow E_1'$ of $\phi_0(\gamma):[0,1]\rightarrow E'$ such that $\alpha(0)=\phi_0(\widetilde{w})$ since $\phi_0\circ f=F\circ\phi_0$ on $\partial D_0=\partial_- E_1$. Define $\phi_{E_1}(z):=\alpha(1)$. We know that $\phi_0\circ f =F\circ \phi_{E_1}$ on $E_1$ and $\phi_{E_1}:E_1\rightarrow E_1'$ is quasiconformal since $f,F$ are both holomorphic covering maps with degree $d_1$ and $\phi_0:E\rightarrow E'$ is quasiconformal. Now some parts of $\phi_1:\overline{\mathbb{C}}\rightarrow\overline{\mathbb{C}}$ are defined as follows: $\phi_1|_{\overline{D}_0}=\phi_0|_{\overline{D}_0}$, $\phi_1|_{\overline{D}_\infty}=\phi_0|_{\overline{D}_\infty}$ and $\phi_1|_{E_1}=\phi_{E_1}$. Then, $\phi_1\circ f =F\circ \phi_1$ on $\partial E_1$. Similarly, there exists a unique quasiconformal mapping $\phi_{E_{2m+1}}:E_{2m+1}\rightarrow E_{2m+1}'$, which is the lift of $\phi_0:E\rightarrow E'$ such that $\phi_0\circ f =F\circ \phi_{E_{2m+1}}$ on $E_{2m+1}$. Define $\phi_1|_{E_{2m+1}}=\phi_{E_{2m+1}}$. Then, $\phi_1\circ f =F\circ \phi_1$ on $\partial E_{2m+1}$.

Unlike the cases of $E_1$ and $E_{2m+1}$, the lift $\phi_{E_i}:E_i\rightarrow E_i'$ of $\phi_0:E\rightarrow E'$ exists but is not unique for $2\leq i\leq 2m$.
We first show the existence of $\phi_{E_i}$. Without loss of generality, suppose that $i$ is even. Since $f:\partial_- E_i\rightarrow\partial D_\infty$ and $F:\partial_- E_i'\rightarrow\partial D_\infty'$ are both covering mappings with degree $d_i$, there exists a lift (not unique) $\phi_{E_i}:\partial_- E_i\rightarrow \partial_- E_i'$ of $\phi_0:\partial D_\infty\rightarrow \partial D_\infty'$ such that $\phi_0\circ f =F\circ \phi_{E_i}$ on $\partial_-E_i$. By using the same method of defining $\phi_{E_1}$, there exists a unique lift of $\phi_0:E\rightarrow E'$ defined from $E_i$ to $E_i'$, which we denote also by $\phi_{E_i}$ such that $\phi_0\circ f =F\circ \phi_{E_i}$ on $E_i$. Note that $\phi_{E_i}:E_i\rightarrow E_i'$ is quasiconformal. Define $\phi_1|_{E_{i}}=\phi_{E_i}$. Then, $\phi_0\circ f =F\circ \phi_1$ on $\bigcup_{i=1}^{2m+1}E_i$ and $\phi_1\circ f =F\circ \phi_1$ on $\bigcup_{i=1}^{2m+1}\partial E_i$.

In order to unify the notations, let $D_{2i-1}:=B_i$ and $D_{2i}:=A_i$ for $1\leq i\leq m$. Then we have $D_i\prec D_{j}$ for $1\leq i<j\leq 2m$. We need to define $\phi_1$ on $\bigcup_{i=1}^{2m}D_i$. For every $D_i$, where $1\leq i\leq 2m$, its two boundary components $\partial_+ E_i$ and $\partial_- E_{i+1}$ are both quasicircles. Since $\phi_{E_{i}}$ and $\phi_{E_{i+1}}$ are both quasiconformal mappings, the map $\phi_1|_{\partial_+ E_{i}\cup\partial_- E_{i+1}}$ has a quasiconformal extension $\phi_{D_i}:\overline{D}_i\rightarrow \overline{D}_i'$ such that $\phi_{D_i}(D_i)=D_i'$. Now we obtain a quasiconformal mapping $\phi_1:\overline{\mathbb{C}}\rightarrow\overline{\mathbb{C}}$ defined as $\phi_1|_{E_i}:=\phi_{E_i}$, $\phi_1|_{D_j}=\phi_{D_j}$ and $\phi_1|_{D_0\cup D_\infty}=\phi_0$, where $1\leq i\leq 2m+1$ and $1\leq j\leq 2m$.

Next, we define $\phi_2$. First, let $\phi_2|_{D_j}=\phi_1$ for $j\in\{0,1,\cdots,2m,\infty\}$. Then we lift $\phi_1:E\rightarrow E'$ in an appropriate way to obtain $\phi_2:E_i\rightarrow E_i'$ for $1\leq i\leq 2m+1$. Finally, we check the continuity of the resulting map $\phi_2:\overline{\mathbb{C}}\rightarrow\overline{\mathbb{C}}$. Now let us make this precise. In order to guarantee the continuity of $\phi_2$ on $D_0\cup E_1$, we need to have $\phi_2|_{\partial_-E_1}=\phi_1$. Then there exists only one way to lift $\phi_1:E\rightarrow E'$ to obtain $\phi_2:E_1\rightarrow E_1'$. In order to guarantee the continuity of the lift $\phi_2$, we need to check the continuity of $\phi_2$ on the boundary $\partial_+E_1$ first. In fact, $\phi_0|_E$ and $\phi_1|_E$ are homotopic to each other and $\phi_1|_{\partial E}=\phi_0|_{\partial E}$, it follows that $\phi_2|_{\partial_+ E_1}=\phi_1|_{\partial_+ E_1}$ since $\phi_2|_{\partial_- E_1}=\phi_1|_{\partial_- E_1}$. This means that $\phi_2$ is continuous on $\partial_+ E_1$. Similarly, we can lift $\phi_1:E\rightarrow E'$ to obtain $\phi_2:E_i\rightarrow E_i'$ for $2\leq i\leq 2m+1$ and guarantee the continuity of $\phi_2$. Above all, the map $\phi_2:\overline{\mathbb{C}}\rightarrow\overline{\mathbb{C}}$ satisfies (1) $\phi_2$ is quasiconformal and the dilatation $K(\phi_2)=K(\phi_1)$; (2) $\phi_2|_{f^{-1}(D_0\cup D_\infty)}=\phi_1$; (3) $\phi_1\circ f=F\circ\phi_2$ on $\bigcup_{i=1}^{2m+1}E_i$ and hence $\phi_2\circ f=F\circ\phi_2$ on $f^{-2}(\partial D_0\cup \partial D_\infty)$.

Suppose we have obtained $\phi_k$ for some $k\geq 1$, then $\phi_{k+1}$ can be defined completely similarly to the process of the derivation of $\phi_2$ from $\phi_1$. Inductively, we can obtain a sequence of quasiconformal mappings $\{\phi_k\}_{k\geq 0}$ such that (1) $K(\phi_k)=K(\phi_1)\geq K(\phi_0)$ for $k\geq 1$; (2) $\phi_{k+1}(z)=\phi_{k}(z)$ for $z\in f^{-k}(D_0\cup D_\infty)$; (3) $\phi_k\circ f=F\circ\phi_k$ on $f^{-k}(\partial D_0\cup \partial D_\infty)$. This means that $\{\phi_k\}_{k\geq 0}$ forms a normal family. Take a convergent subsequence of $\{\phi_k\}_{k\geq 0}$ whose limit we denote by $\phi_\infty$, then $\phi_\infty$ is a quasiconformal mapping satisfying $\phi_\infty\circ f=F\circ\phi_\infty$ on $\bigcup_{k\geq 0} f^{-k}(\partial D_0\cup \partial D_\infty)$. Moreover, $K(\phi_\infty)\leq K(\phi_1)$. Since $\phi_\infty$ is continuous, $\phi_\infty\circ f=F\circ\phi_\infty$ holds on the closure of $\bigcup_{k\geq 0}f^{-k}(\partial D_0\cup \partial D_\infty)$, which is the Julia set of $f$. Therefore $\phi=\phi_\infty$ is the quasiconformal mapping we want to find which conjugates $f$ to $F$ on their corresponding Julia sets. This ends the proof of case $f(D_0)=D_0$ and $f(D_\infty)=D_\infty$.

The other three cases: (1) $f(D_0)=D_\infty$, $f(D_\infty)=D_\infty$; (2) $f(D_0)=D_\infty$, $f(D_\infty)=D_0$; and (3) $f(D_0)=D_0$, $f(D_\infty)=D_0$ can be proved completely similarly.

If one or both of the components $D_0$ and $D_\infty$ are parabolic, there exists a perturbation $f_\varepsilon$ of $f$ such that $f_\varepsilon$ is hyperbolic and the dynamics of $f_\varepsilon$ are topologically conjugate to that of $f$ on their corresponding Julia sets \cite{Cui}. Then $f$ has a `model' in \eqref{family} since $f_\varepsilon$ always does. This ends the proof of Theorem \ref{this-is-all-resta} and hence Theorem \ref{this-is-all}.
\end{proof}

From the proof of Theorem \ref{this-is-all-resta} in the hyperbolic case, we have following immediate corollary.

\begin{cor}\label{Julia-comp}
If the parameters $a_i$ are chosen as in Theorem \ref{parameter}, where $1\leq i\leq n-1$, then each Julia component of $f_{p,d_1,\cdots,d_n}$ is a quasicircle.
\end{cor}

\section{Non-hyperbolic rational maps whose Julia sets are Cantor circles}\label{sec-para-mcm}

The rational maps
\begin{equation}
P_\lambda(z)=\frac{\frac{1}{n}((1+z)^n-1)+\lambda^{m+n}z^{m+n}}{1-\lambda^{m+n}z^{m+n}}
\end{equation}
where $\lambda\in\mathbb{C}^*=\mathbb{C}\setminus\{0\}$ and $m,n\geq 2$ are both positive integers satisfying $1/m+1/n<1$ can be seen as a perturbation of the parabolic polynomial
\begin{equation}
\widetilde{P}(z)=\frac{(1+z)^n-1}{n}.
\end{equation}
Note that $\widetilde{P}$ has a parabolic fixed point at the origin with multiplier 1 and critical point $-1$ with multiplicity $n-1$. This means that there exists only one bounded and hence simply connected Fatou component of $\widetilde{P}$ in which all points are attracted to the origin. In particular, the Julia set of $\widetilde{P}$ is a Jordan curve with infinitely many cusps.

We hope that some properties of $\widetilde{P}$ stated above can be also hold for $P_\lambda$ when $\lambda$ is small. But obviously, there are lots of differences between $P_\lambda$ and $\widetilde{P}$. The degree of $P_\lambda$ is $m+n$ and $P_\lambda(\infty)=-1$. There are $2(m+n)-2$ critical points of $P_\lambda$: $m-1$ at $\infty$, $n-1$ are very close to $-1$ and the remaining $m+n$ critical points lie nearby the circle $\mathbb{T}_{r_0/|\lambda|}$, where $r_0=\sqrt[m+n]{n/m}$ (see Lemma \ref{crit-close-para}).
In fact, we will see that $P_\lambda$ can be viewed as a `parabolic' McMullen map at the end of this section since $P_\lambda$ is conjugate to some $g_\eta$ on their corresponding Julia sets.

Firstly, we show that the fixed parabolic Fatou component of $\widetilde{P}$ contains the Euclidean disk $\mathbb{D}(-\frac{3}{4},\frac{3}{4})$ for every $n\geq 2$ and $P_\lambda$ maps $\mathbb{D}(-\frac{3}{4},\frac{3}{4})$ into itself if $\lambda$ is small enough.

\begin{lema}\label{key-lemma}
$(1)$ For every $n\geq 2$, $\widetilde{P}(\overline{\mathbb{D}}(-\frac{3}{4},\frac{3}{4}))\subset\mathbb{D}(-\frac{3}{4},\frac{3}{4})\cup\{0\}$.

$(2)$ If $0<|\lambda|<{1}/{(3n)}$, then $P_\lambda(\overline{\mathbb{D}}(-\frac{3}{4},\frac{3}{4}))\subset\mathbb{D}(-\frac{3}{4},\frac{3}{4})\cup\{0\}$. In particular, $\mathbb{D}(-\frac{3}{4},\frac{3}{4})$ lies in the parabolic Fatou component of $P_\lambda$ with parabolic fixed point $0$.
\end{lema}
\begin{proof}
If $z\in\overline{\mathbb{D}}(-\frac{3}{4},\frac{3}{4})$, then $|\widetilde{P}(z)+1/n|=|1+z|^n/n\leq 1/n$.
In particular, the inequality sign can be replaced by equality if and only if $z=0$. This ends the proof of (1).

The proof of (2) will be divided into two cases: $|z|$ is small and not too small. For every $z=-\frac{3}{4}+\frac{3}{4}e^{i\theta}\in\partial\mathbb{D}(-\frac{3}{4},\frac{3}{4})$, where $-\pi< \theta\leq\pi$, we have $|1+\widetilde{P}(z)|\leq 5/2$ by (1) and $|\lambda z|^{m+n}<1/2$ since $|\lambda|<1/(3n)$. This means that
\begin{equation}
|P_\lambda(z)-\widetilde{P}(z)|=\left|\frac{\lambda^{m+n}z^{m+n}(1+\widetilde{P}(z))}{1-\lambda^{m+n}z^{m+n}}\right|\leq 5|\lambda z|^{m+n}.
\end{equation}
Since $|z|=\frac{3}{4}|1-e^{i\theta}|=\frac{3}{4}|e^{-i\theta/2}-e^{i\theta/2}|=\frac{3}{2}|\sin\frac{\theta}{2}|\leq \frac{3}{4}|\theta|$ and $|\lambda|<1/(3n)$, we have
\begin{equation}\label{P-lambd-est}
|P_\lambda(z)-\widetilde{P}(z)|\leq 5\,({|\theta|}/{(4n)})^{m+n}.
\end{equation}

On the other hand, since $|\sin\theta|\geq \frac{2}{\pi}|\theta|$ if $|\theta|\leq\frac{\pi}{2}$, we have
\begin{equation}\label{P-est-0}
\begin{split}
|\widetilde{P}(z)+{3}/{4}|
=  &~     \left| \frac{(\frac{1}{4}+\frac{3}{4}e^{i\theta})^n-1}{n}+\frac{3}{4}\right|
     \leq \frac{\left|\frac{1}{4}+\frac{3}{4}e^{i\theta}\right|^n-1}{n}+\frac{3}{4} \\
=  &~     \frac{(1-\frac{3}{4}\sin^2\frac{\theta}{2})^{n/2}-1}{n}+\frac{3}{4}
     \leq \frac{(1-\frac{3\theta^2}{4\pi^2})^{n/2}-1}{n}+\frac{3}{4}.
\end{split}
\end{equation}
If $|\theta|<2\pi/n$, then $\frac{3\theta^2}{4\pi^2}<\frac{2}{n}$. By Lemma \ref{very-useful-est}(3), we have
\begin{equation}\label{P-est}
|\widetilde{P}(z)+{3}/{4}|\leq -\frac{\frac{n}{2}\cdot\frac{3\theta^2}{4\pi^2}}{3n}+\frac{3}{4}=\frac{3}{4}-\frac{\theta^2}{8\pi^2}.
\end{equation}
Therefore, combining \eqref{P-lambd-est} and \eqref{P-est}, it follows that if $|\theta|<2\pi/n$, then
\begin{equation}
|P_\lambda(z)+3/4| \leq |\widetilde{P}(z)+{3}/{4}|+|P_\lambda(z)-\widetilde{P}(z)|\leq \frac{3}{4}-\frac{\theta^2}{8\pi^2}+5\,(\frac{|\theta|}{4n})^{m+n}\leq 3/4.
\end{equation}
If $2\pi/n\leq |\theta|\leq \pi$, from \eqref{P-est-0} and \eqref{P-est}, we know that
\begin{equation}\label{P-est-2}
|\widetilde{P}(z)+{3}/{4}|\leq \frac{3}{4}-\frac{1}{2 n^2}.
\end{equation}
From \eqref{P-lambd-est} and \eqref{P-est-2}, it follows that if $2\pi/n\leq |\theta|\leq \pi$, then
\begin{equation}
|P_\lambda(z)+3/4| \leq \frac{3}{4}-\frac{1}{2 n^2}+5\,(\frac{|\theta|}{4n})^{m+n}< 3/4.
\end{equation}

Therefore, we have shown that $|P_\lambda(z)+\frac{3}{4}|\leq \frac{3}{4}$ for every $z\in\partial\mathbb{D}(-\frac{3}{4},\frac{3}{4})$ and $|P_\lambda(z)+\frac{3}{4}|= \frac{3}{4}$ if and only if $z=0$. The proof is complete.
\end{proof}

As in the procedure in \S2, now we locate the free critical points of $P_\lambda$. By a direct calculation, the bounded $m+2n-1$ critical points of $P_\lambda$ are the solutions of
\begin{equation}\label{crit-P-lamb}
(1+z)^{n-1}+\lambda^{m+n}z^{m+n-1}\{(1+m/n)[(1+z)^n+n-1]-z(1+z)^{n-1}\}=0.
\end{equation}

\begin{lema}\label{nice-cond-para}
If $0<|\lambda|<{1}/{(3n)}$, then there are $n-1$ critical points of $P_\lambda$ in $\mathbb{D}(-1,|\lambda|)\subsetneq \mathbb{D}(-\frac{3}{4},\frac{3}{4})$.
\end{lema}

\begin{proof}
If $|z+1|\leq |\lambda|<\frac{1}{3n}$, then $|z|\cdot|1+z|^{n-1}\leq (1+|\lambda|)|\lambda|^{n-1}<1$ and
\begin{equation}
(1+m/n)\,|(1+z)^n+n-1|\leq (1+m/n)(|\lambda|^{n}+n-1)<m+n.
\end{equation}
This means that if $|z+1|\leq |\lambda|$, then
\begin{equation}
\begin{split}
  &~ \left|\lambda^{m+n}z^{m+n-1}\{(1+m/n)[(1+z)^n+n-1]-z(1+z)^{n-1}\}\right|\\
< &~|\lambda|^{n-1}\cdot |\lambda z|^{m-1}|\lambda|^2|z|^n(m+n+1)
< |\lambda|^{n-1}\cdot (2n)^{1-m}(9n^2)^{-1}e^{1/3}(m+n+1)\\
< &~|\lambda|^{n-1}\cdot (m+n-1)/(2n)^{m+1}<|\lambda|^{n-1}.
\end{split}
\end{equation}
By Rouch\'{e}'s Theorem, the proof is completed.
\end{proof}

Let $\widetilde{CP}:=\{\widetilde{w}_{j}=\frac{r_0}{\lambda} \exp(\pi i\frac{2j-1}{m+n}):1\leq j\leq m+n\}$ be the collection of the zeros of $m\lambda^{m+n}z^{m+n}+n=0$, where $r_0=\sqrt[m+n]{n/m}$. Since $h(x)=x^{1/x},x>0$ has maximal value $e^{1/e}<3/2$ at $x=e$, we have
\begin{equation}
2/3<1/\sqrt[m]{m}<r_0<\sqrt[n]{n}<3/2.
\end{equation}
The following lemma shows that the remaining $m+n$ critical points of $P_\lambda$ are very `close' to $\widetilde{CP}$.

\begin{lema}\label{crit-close-para}
If $0<|\lambda|<{1}/{(2^m n^2)}$, then \eqref{crit-P-lamb} has a solution $w_j$ such that $|w_j-\widetilde{w}_j|<2(m+n)/m$, where $1\leq j\leq m+n$. Moreover, $w_i= w_j$ if and only if $i=j$.
\end{lema}

\begin{proof}
Dividing $(1+z)^{n-1}$ on both sides of \eqref{crit-P-lamb}, we have
\begin{equation}
1+\lambda^{m+n}z^{m+n-1}\left(\frac{m}{n}z+\frac{m+n}{n}\left(1+\frac{n-1}{(1+z)^{n-1}}\right)\right)=0.
\end{equation}
Or, in more useful form
\begin{equation}\label{use-ful}
\frac{n}{m\lambda^{m+n}}+z^{m+n}+\frac{(m+n)z^{m+n-1}}{m}\left(1+\frac{n-1}{(1+z)^{n-1}}\right)=0.
\end{equation}

Let $\Omega=\{z:|z^{m+n}+\frac{n}{m}\lambda^{-(m+n)}|\leq \beta|\lambda|\cdot\frac{n}{m}|\lambda|^{-(m+n)}\}$, where $\beta=\frac{2(m+n)}{mr_0}<\frac{3(m+n)}{m}$. If $z\in\Omega$, then $|\lambda^{m+n}z^{m+n}+\frac{n}{m}|<\beta|\lambda|\cdot\frac{n}{m}$ and $|z-\widetilde{w}_j|<\beta r_0$ for some $1\leq j\leq 2n$ by Lemma \ref{very-useful-est}(2). If $z\in\Omega$ and $0<|\lambda|<{1}/{(2^m n^2)}$, we have
\begin{equation}\label{P-estima}
\frac{n-1}{|1+z|^{n-1}}<\frac{n-1}{((|\lambda|^{-1}-\beta)r_0-1)^{n-1}}
<\frac{n-1}{(2^{m+1}n^2/3-3-2n/m)^{n-1}}<\frac{1}{15}
\end{equation}
and
\begin{equation}\label{P-estima-1}
\beta|\lambda|\leq\frac{2(m+n)}{2^m n^2\cdot mr_0}<\frac{3}{2^m n}\left(\frac{1}{m}+\frac{1}{n}\right)<\frac{1}{4}, \text{~~therefore~~}
\frac{1+\beta|\lambda|}{2(1-\beta|\lambda|)}<\frac{5}{6}.
\end{equation}
Therefore, if $z\in\Omega$ and $0<|\lambda|<{1}/{(2^m n^2)}$, from \eqref{P-estima} and \eqref{P-estima-1}, we have
\begin{equation}
\begin{split}
 &~ \left|\frac{(m+n)z^{m+n-1}}{m}\left(1+\frac{n-1}{(1+z)^{n-1}}\right)\right| =
    \frac{m+n}{m|\lambda|^{m+n}}\left|\frac{\lambda^{m+n}z^{m+n}}{z}\left(1+\frac{n-1}{(1+z)^{n-1}}\right)\right|\\
 < &~ \frac{m+n}{m|\lambda|^{m+n}}~\frac{(\beta|\lambda|+1)n/m}{r_0(1/|\lambda|-\beta)}\cdot \frac{16}{15}
    = \frac{n\beta|\lambda|}{m|\lambda|^{m+n}}\,\frac{1+\beta|\lambda|}{2(1-\beta|\lambda|)}\cdot \frac{16}{15}
      <\frac{n\beta|\lambda|}{m|\lambda|^{m+n}}.
\end{split}
\end{equation}
Applying Rouch\'{e}'s Theorem to \eqref{use-ful} and then using Lemma \ref{very-useful-est}(2), the proof of the first assertion is completed. By means of the same argument as \eqref{differ-2}, if $0<|\lambda|<{1}/{(2^m n^2)}$, we have
\begin{equation}
\frac{(r_0/|\lambda|)\cdot\sin(\pi/(m+n))}{2(m+n)/m}\geq \frac{mr_0}{(m+n)^2|\lambda|}>\frac{2^{m+1}m}{3(m/n+1)^2}>1.
\end{equation}
This means that $w_i= w_j$ if and only if $i=j$. The proof is complete.
\end{proof}

Let $CP:=\{w_j:1\leq j\leq m+n\}$ be the $m+n$ critical points of $P_\lambda$ lying near the circle $\mathbb{T}_{r_0/|\lambda|}$ and $CV:=\{P_\lambda(w_j):1\leq j\leq m+n\}$.
Let $CP_{-1}$ be the collection of $n-1$ critical points of $P_\lambda$ near $-1$ (see Lemma \ref{nice-cond-para}) and $CV_{-1}=\{P_\lambda(z):z\in CP_{-1}\}$.

Let $T_0$ be the Fatou component of $P_\lambda$ containing the attracting petal at the origin and $U:=\mathbb{D}(-\frac{3}{4},\frac{3}{4})$. By Lemmas \ref{key-lemma}(2) and \ref{nice-cond-para}, we know that $CP_{-1}\cup CV_{-1}\subset U\subset T_0$.
Since $P_\lambda(\infty)=-1$, it follows that there exists a neighborhood of $\infty$ such that $P_\lambda$ maps it to a neighborhood of $-1$. Let $T_\infty$ be the Fatou component such that $\infty\in T_\infty$ and $U_0,U_\infty$ be the component of $P_\lambda^{-1}(U)$ such that $0\in\overline{U}_0$ and $\infty\in U_\infty$. Obviously, we have $U\subset U_0\subset T_0$ and $U_\infty\subset T_\infty$.

\begin{lema}\label{Nice-cond-para-McM}
If $0<|\lambda|\leq {1}/{(2^{10m} n^3)}$, there exists an annular neighborhood $A_1$ of $CP$ containing $\mathbb{T}_{1/|\lambda|}\cup CP$ such that $P_\lambda(A_1)\subset \overline{U'}_\infty\subset U_\infty$, where $U'_\infty$ is a neighborhood of $\infty$.
\end{lema}

\begin{proof}
It is known from Lemma \ref{crit-close-para} that $CP$ is `almost' lying uniformly on the circle $\mathbb{T}_{r_0/|\lambda|}$ and all the finite poles of $P_\lambda$ lie on the circle $\mathbb{T}_{1/|\lambda|}$. Define the annulus
\begin{equation}
A_1=\{z:1/(2|\lambda|)<|z|<2/|\lambda|\}.
\end{equation}
Note that
\begin{equation}
\frac{r_0}{|\lambda|}+\frac{2(m+n)}{m}<\frac{3}{2|\lambda|}+2+\frac{2n}{m}<\frac{2}{|\lambda|}
\end{equation}
and
\begin{equation}
\frac{r_0}{|\lambda|}-\frac{2(m+n)}{m}>\frac{2}{3|\lambda|}-2-\frac{2n}{m}>\frac{1}{2|\lambda|}.
\end{equation}
We have $\mathbb{T}_{1/|\lambda|}\cup CP\subset A_1$ by Lemma \ref{crit-close-para}. If $z\in A_1$ and $|\lambda|\leq\frac{1}{2^{10m} n^3}$, then
\begin{equation}\label{import-1}
|P_\lambda(z)+1|\geq \frac{(|z|-1)^n}{n(|\lambda z|^{m+n}+1)}\geq\frac{(\frac{1}{2|\lambda|}-1)^n}{n(2^{m+n}+1)}=\frac{(1-2|\lambda|)^n}{2^n n|\lambda|^n(2^{m+n}+1)}>\frac{2}{|\lambda|^{1+\frac{n}{m}}}+1.
\end{equation}
In fact,
\begin{equation}
\frac{(1-2|\lambda|)^n}{2^{m+n}+1}>\frac{(1-\frac{2}{2^{10m}n^3})^n}{2^{m+n}+1}>\frac{0.9}{2^{m+n}+1}>\frac{1}{2^{m+n+1}}+2^n n|\lambda|^n.
\end{equation}
This means that \eqref{import-1} follows by
\begin{equation}
2^{m+2n+2}\,n\,|\lambda|^n\leq |\lambda|^{1+n/m}.
\end{equation}
This is true because $|\lambda|\leq\frac{1}{2^{10m} n^3}$. Now we have proved that if $z\in A_1$ and $|\lambda|\leq\frac{1}{2^{10m} n^3}$, then $|P_\lambda(z)|>\frac{2}{|\lambda|^{1+{n}/{m}}}$.

On the other hand, if $|z|\geq \frac{2}{|\lambda|^{1+{n}/{m}}}$, then
\begin{equation}
|P_\lambda(z)+1|\leq \frac{(|z|+1)^n+1}{|\lambda z|^{m+n}-1}\leq
\frac{(1+|z|^{-1})^n+|z|^{-n}}{2^m-|z|^{-n}}<\frac{1}{2}.
\end{equation}
This means that $P_\lambda(z)\in\mathbb{D}(-1,\frac{1}{2})\subset U$.
Let $U'_\infty$ be the component of $P_\lambda^{-1}(\mathbb{D}(-1,\frac{1}{2}))$ containing $\{z:|z|\geq \frac{2}{|\lambda|^{1+{n}/{m}}}\}$, it follows that $P_\lambda(A_1)\subset \overline{U'}_\infty\subset U_\infty$ (see Figure \ref{Fig_para-map}).
\end{proof}

\begin{figure}[!htpb]
  \setlength{\unitlength}{1mm}
  \centering
  \includegraphics[width=130mm]{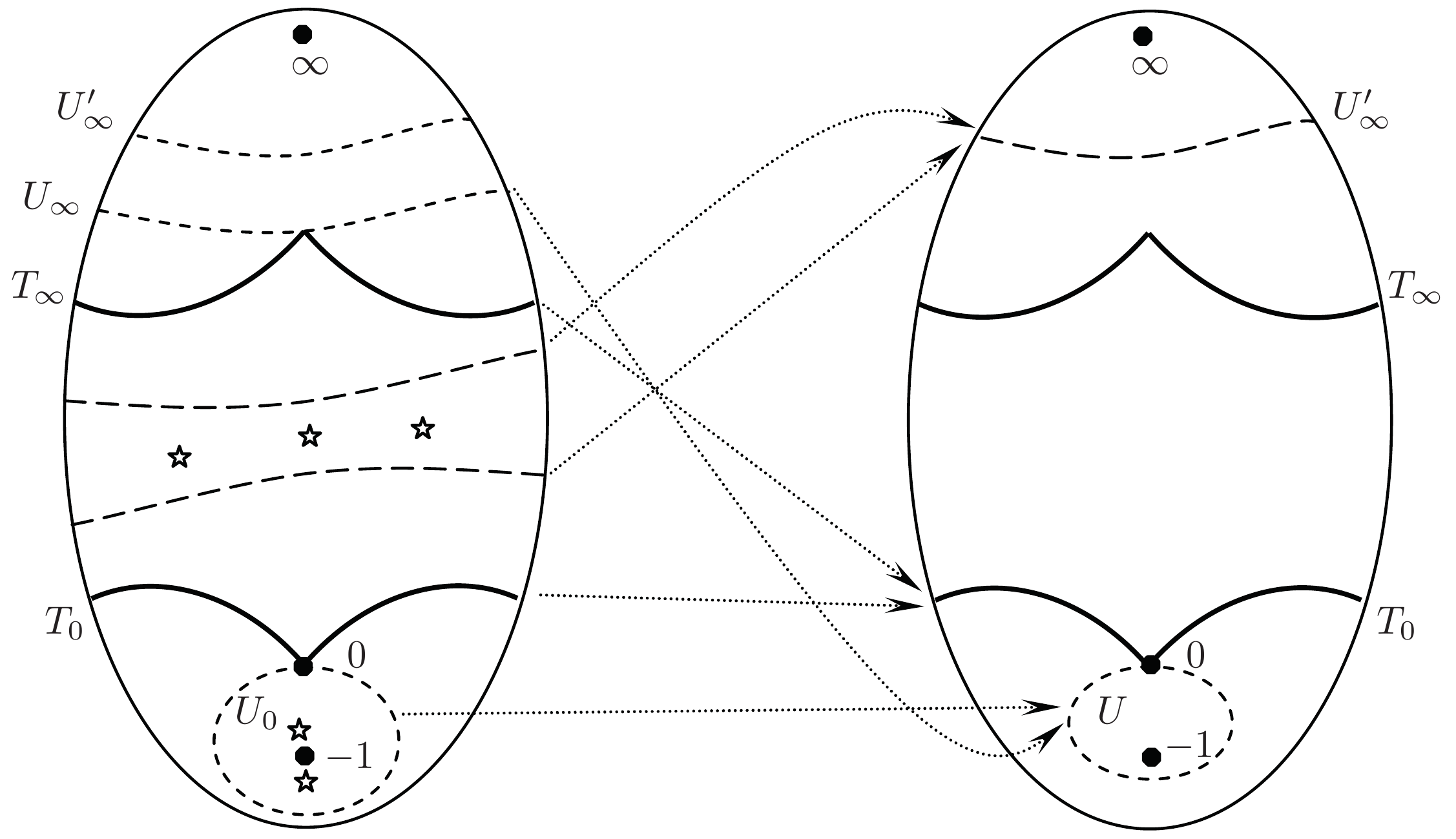}
  \caption{Sketch illustrating of the mapping relation of $P_\lambda$. The small pentagons denote the critical points.}
  \label{Fig_para-map}
\end{figure}

\noindent\textit{Proof of Theorem \ref{non-hyper-cantor}}. For every $\lambda$ such that $0<|\lambda|\leq {1}/{(2^{10m} n^3)}$, let $A:=\overline{\mathbb{C}}\setminus (U\cup U'_\infty)$. Since $P_\lambda:U'_\infty\rightarrow\mathbb{D}(-1,\frac{1}{2})$ is proper with degree $m$, it follows that $U'_\infty$ is simply connected and $A$ is an annulus. Note that $P_\lambda^{-1}(U'_\infty)$ is an annulus since there are $m+n$ critical points in $P_\lambda^{-1}(U'_\infty)$ and on which the degree of $P_\lambda$ is $m+n$. This means that $P_\lambda^{-1}(A)$ consists of two disjoint annuli $I_1$ and $I_2$ and $I_1\cup I_2\subset A$. The degree of the restriction of $P_\lambda$ on $I_1$ and $I_2$ are $m$ and $n$ respectively.

The following argument is very similar to that of Theorem \ref{parameter}. The Julia set of $P_\lambda$ is $J_\lambda=\bigcap_{k\geq 0}P_\lambda^{-k}(A)$. By the construction, the components of $J_n$ are compact sets nested between $-1$ and $\infty$ since $P_\lambda^{-1}:A\rightarrow I_j$ is conformal for $j=1$ or $2$. Since the component of $J_n$ cannot be a point and the proof of Theorem 1.2 in \cite{PT} can also be applied to geometrically finite rational maps (see \cite[$\S$9]{PT} and \cite{TY}), we know that every component of $J_n$ is a Jordan curve. The dynamics of $P_\lambda$ on the set of Julia components is isomorphic to the one-sided shift on $2$ symbols $\Sigma_{2}:=\{0,1\}^{\mathbb{N}}$. In particular, $J_\lambda$ is homeomorphic to $\Sigma_{2}\times\mathbb{S}^1$, which is a Cantor set of circles as claimed.
\hfill $\square$

\begin{rmk}
From the proof of Theorem \ref{non-hyper-cantor} and Theorem \ref{this-is-all-resta}, we know that the dynamics on the Julia set of $P_\lambda$ is conjugate to that of some $g_\eta$ with the form \eqref{McMullen}. Therefore, we can view $P_\lambda$ as a `parabolic' McMullen map since the only difference is the super-attracting basin and its preimages of $g_\eta$ have been replaced by a fixed parabolic basin and its preimages of $P_\lambda$ (see Figure \ref{Fig_C-C-C}).
\end{rmk}

\begin{figure}[!htpb]
  \setlength{\unitlength}{1mm}
  \centering
  \includegraphics[width=80mm]{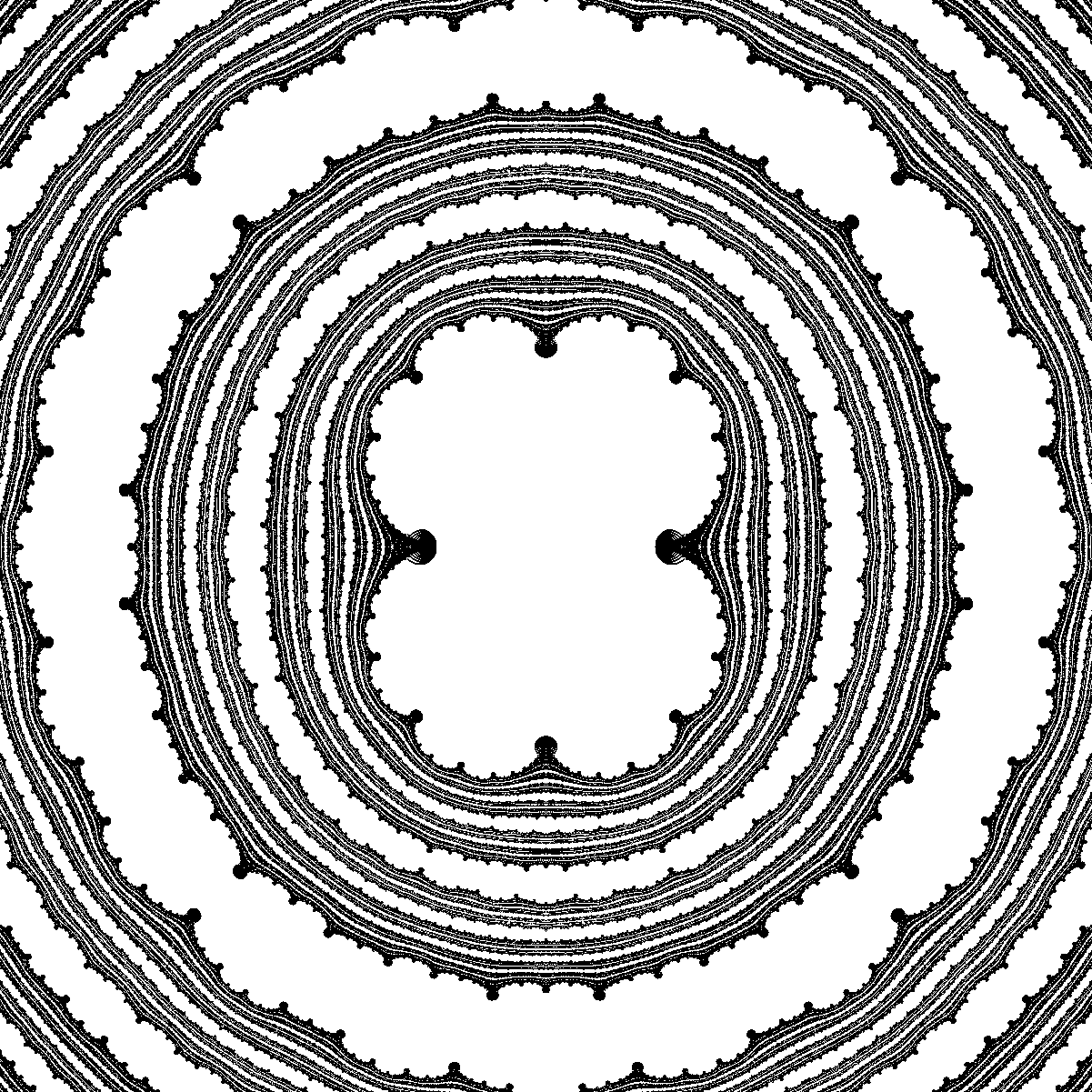}
  \caption{The Julia set of $P_\lambda$, where $m=3,n=2$ and $\lambda$ is small enough such that $J_\lambda$ is a Cantor set of circles. All the Fatou components of $P_\lambda$ are iterated onto the fixed parabolic component (the `cauliflower' in the center of this figure) with parabolic fixed point 1.}
  \label{Fig_C-C-C}
\end{figure}

\section{More Non-hyperbolic Examples}\label{sec-more-exam}

In this section, we will construct more non-hyperbolic rational maps whose Julia sets are Cantor circles but they are not included by the previous section. Inspired by Theorem \ref{parameter}, for every $n\geq 2$, we define
\begin{equation}\label{family-para-restate}
P_n(z)=A_n\,\frac{(n+1)z^{(-1)^{n+1} (n+1)}}{nz^{n+1}+1}\prod_{i=1}^{n-1}(z^{2n+2}-b_i^{2n+2})^{(-1)^{i-1}}+B_n,
\end{equation}
where $|b_i|=s^i$ for some $0<s\leq 1/(25n^2)$ and
\begin{equation}\label{A-B-n-restate}
A_n=\frac{1}{1+(2n+2)C_n}\prod_{i=1}^{n-1}(1-b_i^{2n+2})^{(-1)^i},~~B_n=\frac{(2n+2)C_n}{1+(2n+2)C_n}~~\text{and}~~
C_n=\sum_{i=1}^{n-1}\frac{(-1)^{i-1}b_i^{2n+2}}{1-b_i^{2n+2}}.
\end{equation}

\begin{lema}\label{para-fixed}
$(1)$ $P_n(1)=1$ and $P_n'(1)=1$.

$(2)$ $1-s^{2n+1}/(n+1)<|A_n|<1+s^{2n+1}/(n+1)$ and $|B_n|<s^{2n+1}/(3n+3)$.
\end{lema}
\begin{proof}
It is easy to see $P_n(1)=1$ by a straightforward calculation. Note that
\begin{equation}\label{solu-crit-parabolic}
F_n(z):=\frac{zP_n'(z)}{P_n(z)-B_n}=\sum_{i=1}^{n-1}\frac{(-1)^{i-1}(2n+2)z^{2n+2}}{z^{2n+2}-b_i^{2n+2}}+(-1)^{n+1} (n+1)-\frac{n(n+1)z^{n+1}}{nz^{n+1}+1}.
\end{equation}
This means that
\begin{equation}\label{A-B-equation-1}
\begin{split}
   &~ \frac{P_n'(1)}{P_n(1)-B_n} \\
= &~ (2n+2)\,\sum_{i=1}^{n-1}\frac{(-1)^{i-1}b_i^{2n+2}}{1-b_i^{2n+2}}+(2n+2)\,\sum_{i=1}^{n-1}(-1)^{i-1}+(-1)^{n+1} (n+1)-n\\
= &~ (2n+2)\,\sum_{i=1}^{n-1}\frac{(-1)^{i-1}b_i^{2n+2}}{1-b_i^{2n+2}}+1:=(2n+2)C_n+1.
\end{split}
\end{equation}
Therefore, we have
\begin{equation}\label{A-B-equation-2}
P_n'(1)=(1-B_n)((2n+2)C_n+1)=1.
\end{equation}
It follows that $1$ is a parabolic fixed point of $P_n$. This completes the proof of (1).

For (2), since $|1-b_i^{2n+2}|^{-1}\leq 1+2|b_1|^{2n+2}$ for $1\leq i\leq n-1$ and $0<s\leq 1/(25n^2)\leq 1/100$, then
\begin{equation}\label{C-n-estim}
\begin{split}
      &~ (2n+2)|C_n|< (2n+2)\,(1+2|b_1|^{2n+2})\sum_{i=1}^{n-1}|b_i|^{2n+2}\\
\leq &~ \frac{(2n+2)(1+2s^{2n+2})s^{2n+2}}{1-s^{2n+2}}<\frac{s^{2n+1}}{4n+4}.
\end{split}
\end{equation}
We have
\begin{equation}\label{B-n-estim}
|B_n|=\left|\frac{(2n+2)C_n}{1+(2n+2)C_n}\right|<(2n+2)|C_n|(1+(4n+4)|C_n|)<\frac{s^{2n+1}}{3n+3}
\end{equation}
and
\begin{equation}\label{A-n-estim}
|A_n|<(1+(4n+4)|C_n|)\prod_{i=1}^{n-1}(1+2|b_i|^{2n+2})<(1+\frac{s^{2n+1}}{2n+2})(1+5s^{2n+2})<1+\frac{s^{2n+1}}{n+1}.
\end{equation}
Moreover, we have
\begin{equation}\label{A-n-estim-lower}
|A_n|>(1-(2n+2)|C_n|)\prod_{i=1}^{n-1}(1-|b_i|^{2n+2})>(1-\frac{s^{2n+1}}{4n+4})(1-\frac{s^{2n+2}}{1-s^{2n+2}})>1-\frac{s^{2n+1}}{n+1}.
\end{equation}
The proof is complete.
\end{proof}

Let us first explain some ideas behind the construction. For $n\geq 2$, define $\widetilde{Q}(z)=(z^{n+1}+n)/(n+1)$
and $\varphi(z)=1/z$, then $Q(z):=\varphi\circ\widetilde{Q}\circ\varphi^{-1}(z)=(n+1)z^{n+1}/(nz^{n+1}+1)$ satisfies: $\infty$ is a critical point of $Q$ with multiplicity $n$ which is attracted to the parabolic fixed point $1$. Since $\{b_i\}_{1\leq i\leq n-1}$ are very small, the rational map $P_n$ can be viewed as a small perturbation of $Q$. The terms $A_n$ and $B_n$ here guarantee that $1$ is always a parabolic fixed point of $P_n$ (see Lemma \ref{para-fixed}). It can be shown that $P_n$ maps an annular neighborhood of $\mathbb{T}_{|b_i|}$ into $T_0$ or $T_\infty$ according to whether $i$ is odd or even, where $T_0$ and $T_\infty$ denote the Fatou components containing $0$ and $\infty$ respectively (see Lemma \ref{lemma-want}). The Fatou component $T_\infty$ is always parabolic while $T_0$ is attracting or mapped to $T_\infty$ according to whether $n$ is odd or even. The proof of Theorem \ref{parameter-parabolic} will based on the mixed arguments as in the previous 2 sections.

If $|z|\leq 1$, then $|\widetilde{Q}(z)|\leq 1$. This means that the fixed parabolic Fatou component of $\widetilde{Q}$ contains the unit disk for every $n\geq 2$. Therefore, the parabolic Fatou component of $Q$ contains the exterior of the closed unit disk $\overline{\mathbb{C}}\setminus \overline{\mathbb{D}}$.
Although the polynomial $Q$ has been perturbed into $P_n$, we still have following

\begin{lema}\label{key-lemma-complex}
$P_n(\overline{\mathbb{C}}\setminus \mathbb{D})\subset(\overline{\mathbb{C}}\setminus \overline{\mathbb{D}})\cup\{1\}$. In particular, the disk $\overline{\mathbb{C}}\setminus\overline{\mathbb{D}}$ lies in the parabolic Fatou component of $P_n$ with parabolic fixed point $1$.
\end{lema}

The proof of Lemma \ref{key-lemma-complex} is very subtle, and will be delayed to next section.

\begin{lema}\label{sum-of}
For $n\geq 2$ and $1\leq i\leq n-1$, then
\begin{equation}\label{sum-equa}
\sum_{1\leq j<i}(-1)^j +\sum_{i<j\leq n-1}(-1)^{j-1}+\frac{1+(-1)^{n+1}}{2}=0.
\end{equation}
\end{lema}
\begin{proof}
The argument is based on several cases shown in Table \ref{Tab_number}.
\end{proof}
%\vskip-0.5cm
\begin{table}[htpb]
\renewcommand{\arraystretch}{1.2}
\begin{center}
\begin{tabular}{|c|c|c|c|c|}
\hline
          &          &  $\sum_{1\leqslant j<i}(-1)^j$  & $\sum_{i<j\leqslant n-1}(-1)^{j-1}$  & $(1+(-1)^{n+1})/{2}$\\ \hline
 odd $n$  & odd $i$  &  $0$                         & $-1$                              & $1$         \\ \cline{2-5}
          & even $i$ &  $-1$                        & $0$                               & $1$         \\ \hline
 even $n$ & odd $i$  &  $0$                         & $0$                               & $0$         \\ \cline{2-5}
          & even $i$ &  $-1$                        & $1$                               & $0$         \\ \hline
\end{tabular}
 \vskip0.2cm
 \caption{The proof of Lemma \ref{sum-of}.}
 \label{Tab_number}
\end{center}
\end{table}

\vskip-0.5cm
As before, we first locate the critical points of $P_n$. Note that $0$ and $\infty$ are both critical points of $P_n$ with multiplicity $n$ and the degree of $P_n$ is $n^2+n$. The remaining $2(n^2-1)$ critical points of $P_n$ are the solutions of $F_n(z)=0$ (see equation \eqref{solu-crit-parabolic}).

For $1\leq i\leq n-1$, let $\widetilde{CP}_i:=\{\widetilde{w}_{i,j}=b_i \exp(\pi \textup{i}\frac{2j-1}{2n+2}):1\leq j\leq 2n+2\}$ be the collection of $2n+2$ points lying on $\mathbb{T}_{|b_i|}$ uniformly.
The following lemma is similar to Lemmas \ref{crit-close} and \ref{crit-close-para}.

\begin{lema}\label{crit-close-Parameter}
For every $\widetilde{w}_{i,j}\in\widetilde{CP}_i$, where $1\leq i\leq n-1$ and $1\leq j\leq 2n+2$, there exists $w_{i,j}$, which is a solution of $F_n(z)=0$, such that $|w_{i,j}-\widetilde{w}_{i,j}|<s^{n+1/2}|b_i|$. Moreover, $w_{i_1,j_1}= w_{i_2,j_2}$ if and only if $(i_1,j_1)=(i_2,j_2)$.
\end{lema}

\begin{proof}
Note that $F_n(z)=0$ is equivalent to
\begin{equation}\label{solu-crit-2}
\sum_{i=1}^{n-1}(-1)^{i-1}\frac{z^{2n+2}+b_i^{2n+2}}{z^{2n+2}-b_i^{2n+2}}+\frac{1+(-1)^{n+1}}{2}-\frac{nz^{n+1}}{nz^{n+1}+1}=0.
\end{equation}
Timing $z^{2n+2}-b_i^{2n+2}$ on both sides of (\ref{solu-crit-2}), where $1\leq i\leq n-1$, we have
\begin{equation}\label{solu-crit-3-parabolic}
(-1)^{i-1}(z^{2n+2}+b_i^{2n+2})+(z^{2n+2}-b_i^{2n+2})\,G_{i}(z)=0,
\end{equation}
where
\begin{equation}\label{G_n-parabolic}
G_{i}(z)=\sum_{1\leq j \leq n-1,\,j\neq i}(-1)^{j-1}\frac{z^{2n+2}+b_j^{2n+2}}{z^{2n+2}-b_j^{2n+2}}+\frac{1+(-1)^{n+1}}{2}-\frac{nz^{n+1}}{nz^{n+1}+1}.
\end{equation}

Let $\Omega_{i}=\{z:|z^{2n+2}+b_i^{2n+2}|\leq s^{n+1/2}|b_i|^{2n+2}\}$, where $1\leq i\leq n-1$. If $z\in\Omega_i$, then $|z|^{n+1}\leq (1+s^{n+1/2})|b_i|^{n+1}\leq(1+s^{n+1/2})s^{n+1}$ by Lemma \ref{very-useful-est}(2). So
\begin{equation*}
\left|\frac{nz^{n+1}}{nz^{n+1}+1}\right|\leq \frac{n(1+s^{n+1/2})s^{n+1}}{1-n(1+s^{n+1/2})s^{n+1}}
\leq \frac{(1+100^{-5/2})s^{n+1/2}/5}{1-(1+100^{-5/2})100^{-5/2}/5}<0.3 \, s^{n+1/2}
\end{equation*}
since $s\leq 1/(25n^2)\leq 1/100$. For every $z\in\Omega_{i}$, if $1\leq j<i$, we have
\begin{equation}\label{estim-1-new}
|{z}/{b_j}|^{2n+2}=|{z}/{b_i}|^{2n+2}|{b_i}/{b_j}|^{2n+2}< (1+s^{n+1/2})\,s^{(2n+2)(i-j)}.
\end{equation}
If $i<j\leq n-1$, by the first statement of Lemma \ref{very-useful-est}(2), we have
\begin{equation}\label{estim-2-new}
|{b_j}/{z}|^{2n+2}=|{b_i}/{z}|^{2n+2}|{b_j}/{b_i}|^{2n+2}\leq (1+2\cdot s^{n+1/2})\,s^{(2n+2)(j-i)}.
\end{equation}
From (\ref{estim-1-new}), (\ref{estim-2-new}) and Lemma \ref{sum-of}, we have
\begin{equation}\label{bound-neww}
\begin{split}
   &~ \left|G_{i}(z)+\frac{nz^{n+1}}{nz^{n+1}+1}\right|\\
 = &~ \left|\sum_{1\leq j<i}(-1)^{j}\frac{1+(z/b_j)^{2n+2}}{1-(z/b_j)^{2n+2}}+
               \sum_{i< j\leq n-1}(-1)^{j-1}\frac{1+(b_j/z)^{2n+2}}{1-(b_j/z)^{2n+2}}+\frac{1+(-1)^{n+1}}{2}\right|\\
 < &~ 3\cdot(1+2\cdot s^{n+1/2})\,\left(\sum_{1\leq j<i}s^{(2n+2)(i-j)}+\sum_{i< j\leq n-1} s^{(2n+2)(j-i)}\right)\\
 < &~6\cdot(1+2\cdot s^{n+1/2})^2\,s^{2n+2}.
\end{split}
\end{equation}
The first inequality in (\ref{bound-neww}) follows from the inequality $2x/(1-x)\leq 3x$ if $x<1/3$ (Here $x\leq (1+2\cdot s^{n+1/2})\,s^{2n+2}<10^{-10}$). So we have
\begin{equation}\label{bound-parabolic}
|G_{i}(z)|< ~6\cdot(1+2\cdot s^{n+1/2})^2\,s^{2n+2}+0.3 \, s^{n+1/2} < 0.4 \,s^{n+1/2}.
\end{equation}
Therefore, if $z\in\Omega_{i}$, then
\begin{equation}
|z^{2n+2}-b_i^{2n+2}|\cdot|\,G_{i}(z)|< (2+s^{n+1/2})|b_i|^{2n+2}\cdot 0.4 \,s^{n+1/2}  <  s^{n+1/2}|b_i|^{2n+2}.
\end{equation}
From (\ref{solu-crit-3-parabolic}) and by Rouch\'{e}'s Theorem, there exists a solution $w_{i,j}$ of $F_n(z)=0$ such that $w_{i,j}\in\Omega_i$ for every $1\leq j\leq 2n+2$. In particular, $|w_{i,j}-\widetilde{w}_{i,j}|<s^{n+1/2}|b_i|$ by the second statement of Lemma \ref{very-useful-est}(2). The assertion $w_{i_1,j_1}= w_{i_2,j_2}$ if and only if $(i_1,j_1)=(i_2,j_2)$ can be verified similarly as \eqref{differ-1} and \eqref{differ-2}. The proof is complete.
\end{proof}

For $1\leq i\leq n-1$, let $CP_i:=\{w_{i,j}: 1\leq j\leq 2n+2\}$ be the collection of critical points of $P_n$ which lie close to the circle $\mathbb{T}_{|b_i|}$.

\begin{lema}\label{lemma-want}
There exist $n-1$ annuli $\{A_i\}_{i=1}^{n-1}$ satisfying $A_{n-1}\prec \cdots\prec A_1$ and two simply connected domain $U_0$ and $U_\infty$ which contains $0$ and $\infty$ respectively, such that

$(1)$ $U_\infty\supset\overline{\mathbb{C}}\setminus\overline{\mathbb{D}}$ and $P_n(\overline{U}_\infty)\subset U_\infty\cup\{1\}$;

$(2)$ $A_i\supset\mathbb{T}_{|b_i|}\cup CP_i$, $P_n(\overline{A}_i)\subset U_0$ for odd $i$ and $P_n(\overline{A}_i)\subset U_\infty$ for even $i$;

$(3)$ $P_n(\overline{U}_0)\subset U_\infty$ for even $n$ and $P_n(\overline{U}_0)\subset U_0$ for odd $n$.
\end{lema}

\begin{proof}
Let $U_\infty:=\overline{\mathbb{C}}\setminus\overline{\mathbb{D}}$ be the exterior of the closed unit disk. Then (1) is obvious if we apply Lemma \ref{key-lemma-complex}. Let $\varepsilon=s^{n+1/2}$ and $A_i=\mathbb{A}_{|b_i|(1-2\varepsilon),|b_i|(1+2\varepsilon)}$. From \eqref{family-para-restate}, we know that
\begin{equation}\label{abs-f-n-yang}
|R_n(z)|: = \left|\frac{P_n(z)-B_n}{A_n}\cdot\frac{nz^{n+1}+1}{n+1}\right| = |z|^{(-1)^{n+1} (n+1)}\,|z^{2n+2}-b_i^{2n+2}|^{(-1)^{i-1}}H_i(z),
\end{equation}
where
\begin{equation}\label{H-i-yang}
H_i(z)=\prod_{j=1}^{i-1}|b_j|^{(2n+2)(-1)^{j-1}}
       \prod_{j=i+1}^{n-1}|z|^{(2n+2)(-1)^{j-1}}\cdot Q_i(z)
\end{equation}
and
\begin{equation}\label{Q-i-yang}
Q_i(z)=\prod_{j=1}^{i-1}\left|1-({z}/{b_j})^{2n+2}\right|^{(-1)^{j-1}}
       \prod_{j=i+1}^{n-1}\left|1-({b_j}/{z})^{2n+2}\right|^{(-1)^{j-1}}.
\end{equation}

If $z\in \overline{A}_i$, where $1\leq i\leq n-1$, we have
\begin{equation}\label{Q-i-esti-1-yang}
Q_i(z)<  \prod_{j=1}^{i-1}\left(1+3\,|{b_i}/{b_j}|^{2n+2}\right)
            \prod_{j=i+1}^{n-1}\left(1+3\,|{b_j}/{b_i}|^{2n+2}\right)
        <   (1+6s^{2n+2})^2
\end{equation}
and
\begin{equation}\label{Q-i-esti-2-yang}
Q_i(z)>  \prod_{j=1}^{i-1}\left(1+3\,|{b_i}/{b_j}|^{2n+2}\right)^{-1}
            \prod_{j=i+1}^{n-1}\left(1+3\,|{b_j}/{b_i}|^{2n+2}\right)^{-1}
        >   (1+6s^{2n+2})^{-2}.
\end{equation}

Note that $\varepsilon=s^{n+1/2}\leq (5n)^{-2n-1}\leq 10^{-5}$. If $n$ is even and $1\leq i\leq n-1$ is odd, then for $z\in \overline{A}_i$, we have
\begin{equation*}
\begin{split}
|R_n(z)|
 =    &~ \frac{|z^{2n+2}-b_i^{2n+2}|}{|z|^{n+1}}\,\frac{1}{s^{(i-1)(n+1)}}\, Q_i(z)
 <    \frac{|b_i|^{n+1}(1+(1+2\varepsilon)^{2n+2})}{(1-2\varepsilon)^{n+1}}\,\frac{(1+6s^{2n+2})^2}{s^{(i-1)(n+1)}}\\
 = &~\frac{1+(1+2\varepsilon)^{2n+2}}{(1-2\varepsilon)^{n+1}}(1+6s^{2n+2})^2 s^{n+1}<2.1\cdot s^{n+1}.
\end{split}
\end{equation*}
If $n$ and $1\leq i\leq n-1$ are both even, then for $z\in \overline{A}_i$, we have
\begin{equation*}
|R_n(z)|=\frac{|b_{i-1}|^{2n+2}|z|^{2n+2}}{|z|^{n+1}|z^{2n+2}-b_i^{2n+2}|}\,\frac{1}{s^{(i-2)(n+1)}}\, Q_i(z)
 >     \frac{(1-2\varepsilon)^{n+1}}{1+(1+2\varepsilon)^{2n+2}}\,(1-6s^{2n+2})^2 > 0.49.
\end{equation*}
This means that if $n$ is even and $1\leq i\leq n-1$ is odd, for $z\in \overline{A}_i$, we have
\begin{equation*}
\begin{split}
      &~ |P_n(z)|<\left|\frac{2.1\cdot s^{n+1}\cdot (n+1)\,A_n}{nz^{n+1}+1}\right|+|B_n| \\
\leq &~ \frac{2.1\,(s^{n+1/2}/5)\cdot(1+s^{2n+1}/(n+1))}{1-n(1+2\varepsilon)s^{n+1}}+\frac{s^{2n+1}}{3n+3}<s^{n+1/2}
\end{split}
\end{equation*}
by Lemma \ref{para-fixed}(2).
If $n$ and $1\leq i\leq n-1$ are both even, then for $z\in \overline{A}_i$, we have
\begin{equation*}
\begin{split}
|P_n(z)| 
>     &~ \left|\frac{0.49(n+1)A_n}{nz^{n+1}+1}\right|-|B_n| \\
\geq &~ \frac{0.49(n+1)(1-s^{2n+1}/(n+1))}{1+n(1+2\varepsilon)s^{n+1}}-\frac{s^{2n+1}}{3n+3}>\frac{n+1}{3}\geq 1.
\end{split}
\end{equation*}

By the completely similar arguments, one can show that if $n$ is odd, for $z\in \overline{A}_i$, we have
\begin{equation}\label{bound-f-n-3-parabolic}
|P_n(z)|<s^{n+1/2} \text{~for odd~} i \text{~and~} |P_n(z)|>1 \text{~for even~} i.
\end{equation}
Let $U_0=\mathbb{D}_r$, where $r=s^{n+1/2}$. This proves (2).

If $n$ is odd, for every $z$ such that $|z|\leq  s^{n+1/2}$, we have
\begin{equation*}
\begin{split}
|P_n(z)| \leq
     & \left|\frac{(n+1)A_n}{nz^{n+1}+1}\right|\,|z|^{n+1}  \prod_{i=1}^{n-1}|b_i|^{(2n+2)(-1)^{i-1}} \prod_{i=1}^{n-1}\left|1-\frac{z^{2n+2}}{b_i^{2n+2}}\right|^{(-1)^{i-1}}+|B_n|\\
   \leq & \frac{(n+1)(1+s^{2n+1}/(n+1))}{1-ns^{n^2+n/2}}\,s^{3(n+1)/2}\prod_{i=1}^{n-1}\left(1+2\frac{|z|^{2n+2}}{|b_i|^{2n+2}}\right)+\frac{s^{2n+1}}{3n+3}<s^{n+1/2}.
\end{split}
\end{equation*}
It follows that $P_n(\overline{\mathbb{D}}_r)\subset\mathbb{D}_r$ for odd $n$, where $r=s^{n+1/2}$.

If $n$ is even, then $P_n$ maps a neighborhood of $0$ to that of $\infty$. For every $z$ such that $|z|\leq  s^{n+1/2}$, we have
\begin{equation}\label{bound-lower-in-disk-even-parabolic-lp}
\begin{split}
|P_n(z)| 
\geq &~\frac{(n+1)\,s^{-(n+1)/2}\,(1-s^{2n+1}/(n+1))}{1+ns^{n^2+n/2}}
            \prod_{i=1}^{n-1}\left(1-2\frac{|z|^{2n+2}}{|b_i|^{2n+2}}\right)-\frac{s^{2n+1}}{3n+3} \\
>     &~ n>1.
\end{split}
\end{equation}
This ends the proof of (3). The proof is complete.
\end{proof}

\begin{figure}[!htpb]
  \setlength{\unitlength}{1mm}
  \centering
  \includegraphics[width=100mm]{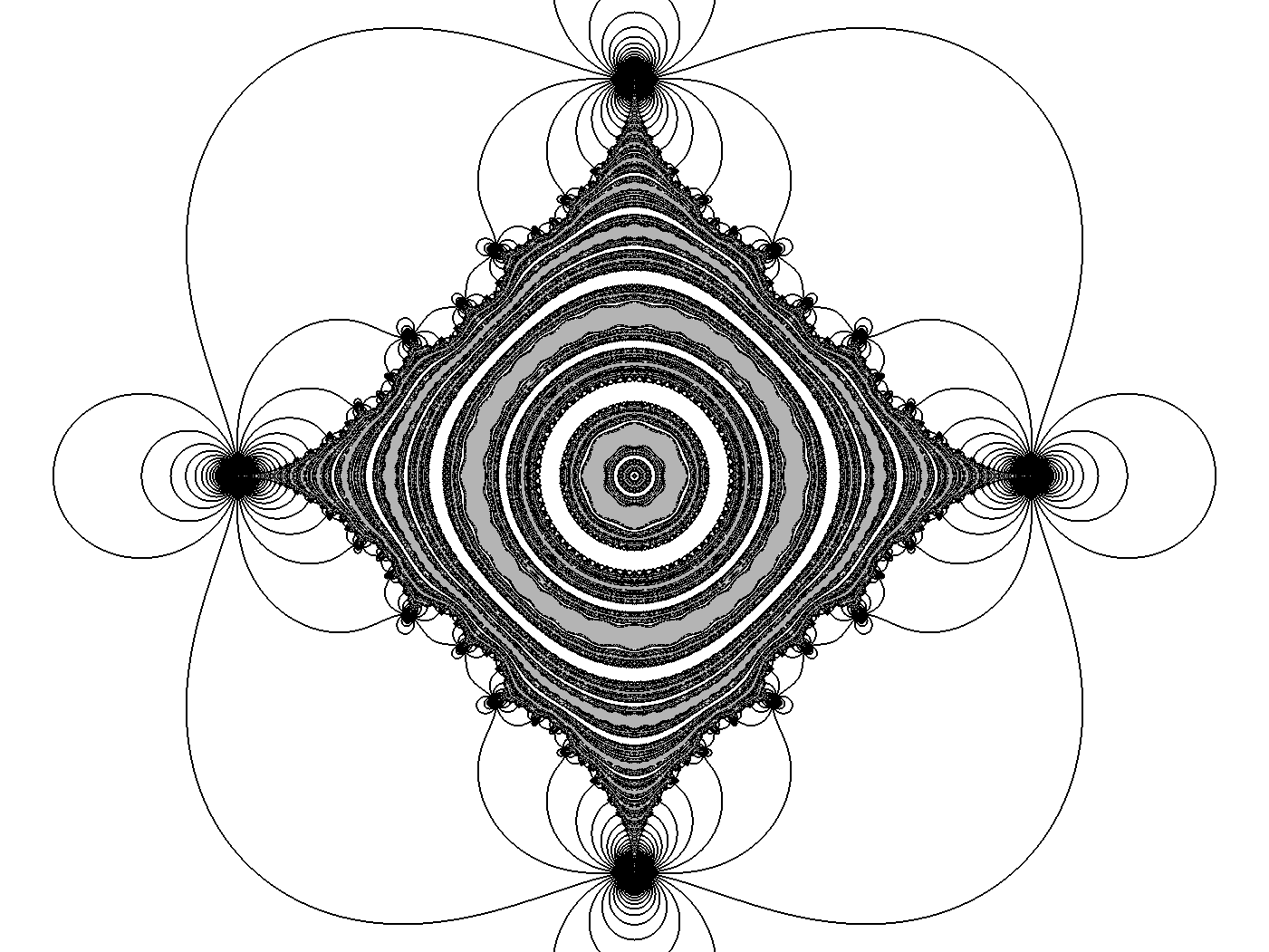}
  \caption{The Julia set of $P_3$, which is a Cantor set of circles. The parameter $s$ is chosen small enough. The gray parts in the Figure denote the Fatou components which are iterated to the attracting Fatou component containing the origin, while the white parts denote the Fatou components iterated to the parabolic Fatou component whose boundary contains the parabolic fixed point $1$. Some equipotentials of Fatou coordinate have been drawn in the parabolic Fatou component and its preimages. Figure range: $[-1.6,1.6]\times[-1.2,1.2]$.}
  \label{Fig_C-C-F}
\end{figure}

\noindent\textit{Proof of Theorem \ref{parameter-parabolic}}. Let $A:=\overline{\mathbb{C}}\setminus (U_0\cup U_\infty)$. The Julia set of $P_n$ is equal to $\bigcap_{k\geq 0}P_n^{-k}(A)$. Note that $P_n$ is geometrically finite. The argument is completely similar to the proofs of Theorems \ref{parameter} and \ref{non-hyper-cantor}. The set of Julia components of $P_n$ is isomorphic to the one-sided shift on $n$ symbols $\Sigma_{n}:=\{0,1,\cdots,n-1\}^{\mathbb{N}}$. In particular, the Julia set of $P_n$ is homeomorphic to $\Sigma_{n}\times\mathbb{S}^1$, which is a Cantor set of circles, as desired (see Figure \ref{Fig_C-C-F}). We omit the details here.
\hfill $\square$

\section{Proof of Lemma \ref{key-lemma-complex}}\label{sec-key-lemma}

This section will be devote to proving Lemma \ref{key-lemma-complex}, which is the key ingredient in the proof of Lemma \ref{lemma-want} and hence in Theorem \ref{parameter-parabolic}.
\begin{proof}
Let $\widetilde{R}(z)=1/P_n(1/z)$, then Lemma \ref{key-lemma-complex} reduces to proving $\widetilde{R}(\overline{\mathbb{D}})\subset \mathbb{D}\cup\{1\}$.
Let $w=z^{n+1}$, by a straightforward calculation, we have
\begin{equation}
R(w):=\widetilde{R}(z)=\frac{w+n}{n+1}\cdot\frac{1}{S(w)},
\end{equation}
where
\begin{equation}\label{S-w-ori}
\begin{split}
S(w)= &~ A_n\,\prod_{i=1}^{n-1}(1-b_i^{2n+2}w^2)^{(-1)^{i-1}}+\frac{w+n}{n+1}B_n \\
      = &~ 1+\frac{w-1}{1+(2n+2)C_n}\left(\frac{H(w)-1}{w-1}+2C_n\right)
\end{split}
\end{equation}
and
\begin{equation}
H(w)=\prod_{i=1}^{n-1}(1-b_i^{2n+2})^{(-1)^{i}}\prod_{i=1}^{n-1}(1-b_i^{2n+2}w^2)^{(-1)^{i-1}}.
\end{equation}

Since $H(1)=1$, it follows that $H'(1)$ is a finite number. In fact,
\begin{equation}\label{I_w}
I(w):=\frac{H'(w)}{H(w)}=-2w\,\sum_{i=1}^{n-1}\frac{(-1)^{i-1}b_i^{2n+2}}{1-b_i^{2n+2}w^2}.
\end{equation}
We know that $I(1)=H'(1)=-2C_n$. For every small enough $w-1$, we can write $S(w)$ as
\begin{equation}\label{S-w}
S(w)=1+\frac{(w-1)^2}{1+(2n+2)C_n}\cdot \frac{\frac{H(w)-1}{w-1}+2C_n}{w-1}=:1+\frac{(w-1)^2}{1+(2n+2)C_n}\cdot\Phi(w),
\end{equation}
where
\begin{equation}\label{Phi-w}
\Phi(w)=\sum_{k\geq 2}\frac{H^{(k)}(1)}{k!}(w-1)^{k-2}.
\end{equation}
The next step is to estimate $H^{(k)}(1)$ for every $k\geq 2$.

For every $k\geq 1$, let
\begin{equation}
Y_k(w)=\sum_{i=1}^{n-1}(-1)^{i-1}\left(\frac{b_i^{2n+2}}{1-b_i^{2n+2}w^2}\right)^k.
\end{equation}
In particular, $Y_1(1)=C_n$ and
\begin{equation}
Y_k'(w)=2kw\,Y_{k+1}(w).
\end{equation}
If $|w|=1$, we have
\begin{equation}
|Y_k(w)|\leq \left|\frac{b_1^{2n+2}}{1-b_1^{2n+2}}\right|^k
\left(1+\sum_{i=2}^{n-1}\left|\frac{b_i^{2n+2}(1-b_1^{2n+2})}{b_1^{2n+2}(1-b_i^{2n+2})}\right|^k\right)
\leq \frac{11}{10}\,\left|\frac{b_1^{2n+2}}{1-b_1^{2n+2}}\right|^k.
\end{equation}
Similarly, we have $|Y_k(w)|\geq \frac{9}{10}|{b_1^{2n+2}}/{(1-b_1^{2n+2})}|^k$. This means that
\begin{equation}\label{Y_k}
\left|\frac{Y_{k+1}(w)}{Y_k(w)}\right|\leq \frac{11}{9}\left|\frac{b_1^{2n+2}}{1-b_1^{2n+2}}\right|\leq 2s^{2n+2}<1/2.
\end{equation}

We first claim that $|I^{(k)}(1)|\leq 2^{k+1}k!|C_n|$ for every $k\geq 0$. Since $I^{(0)}(w)=-2wY_1(w)$ and $I^{(1)}(w)=-2Y_1(w)-4w^2Y_2(w)$, it can be proved inductively that $I^{(k)}(w)$ can be written as
\begin{equation}\label{expansion}
I^{(k)}(w)=\sum_{j=1}^{2^k}Q_{k,j}(w)=\sum_{j=1}^{2^k}P_{k,j}(w)Y_{k,j}(w),
\end{equation}
where $P_{k,j}(w)$ is a polynomial with degree at most $k+1$ and $Y_{k,j}=Y_l$ for some $1\leq l\leq k+1$. Note that some terms $Q_{k,j}$ may be equal to zero (the degree of corresponding polynomial $P_{k,j}$ is regarded as $-\infty$) and the formula \eqref{expansion} can be simplified, but what we need is this `long' expansion. In particular, without loss of generality, for $1\leq j\leq 2^k$, we require further that
\begin{equation}\label{dera}
P_{k+1,2j-1}(w)Y_{k+1,2j-1}(w)=P_{k,j}'(w)Y_{k,j}(w)~~\text{and}~~P_{k+1,2j}(w)Y_{k+1,2j}(w)=P_{k,j}(w)Y_{k,j}'(w).
\end{equation}
Since $\deg (P_{k,j})\leq k+1$ and $Y_{k,j}=Y_l$ for some $1\leq l\leq k+1$, it follows that
\begin{equation}\label{deri-leq}
\begin{split}
     &~ |P_{k+1,2j-1}(1)Y_{k+1,2j-1}(1)|+|P_{k+1,2j}(1)Y_{k+1,2j}(1)|\\
=    &~ |P_{k,j}'(1)Y_{l}(1)|+|P_{k,j}(1)Y_{l}'(1)|\\
\leq &~ (k+1)|P_{k,j}(1)Y_{l}(1)|+2(k+1)|P_{k,j}(1)Y_{l+1}(1)|\\
\leq &~ 2(k+1)|P_{k,j}(1)Y_{k,j}(1)|
\end{split}
\end{equation}
since $|Y_{l+1}(1)/Y_{l}(1)|\leq 1/2$ for every $l\geq 1$ by \eqref{Y_k}.

Denote $||I^{(k)}(1)||:=\sum_{j=1}^{2^k}|P_{k,j}(1)Y_{k,j}(1)|$, we have $||I^{(k)}(1)||\leq 2k||I^{(k-1)}(1)||$. This means that
\begin{equation}\label{bound-I-k}
|I^{(k)}(1)|\leq ||I^{(k)}(1)||\leq 2^k k!||I^{(0)}(1)||=2^{k+1}k!|C_n|.
\end{equation}
This proves the claim $|I^{(k)}(1)|\leq 2^{k+1}k!|C_n|$ for every $k\geq 0$.

Secondly, we check by induction that $|H^{(k)}(1)|\leq 4^k k!|C_n|$ for $k\geq 1$. For $k=1$, we have $|H'(1)|=2|C_n|< 4|C_n|$. Assume that $|H^{(i)}(1)|\leq 4^i i!|C_n|$ for every $1\leq i\leq k$. By \eqref{I_w}, we have $H'(w)=H(w)I(w)$. So
\begin{equation}\label{Deri-H-k}
\begin{split}
|H^{(k+1)}(1)|
\leq &~ |I^{(k)}(1)|+\sum_{i=1}^{k}\frac{k!}{i!(k-i)!}|H^{(i)}(1)|\cdot|I^{(k-i)}(1)|\\
\leq &~ 2^{k+1}k!|C_n|(1+2^{k+1}|C_n|)\leq 4^{k+1} (k+1)!|C_n|
\end{split}
\end{equation}
since $|I^{(k-i)}(1)|\leq 2^{k-i+1}(k-i)!|C_n|$ and $|H^{(i)}(1)|\leq 4^i i!|C_n|$ for every $1\leq i\leq k$.

If $w=e^{i\theta}$ for $|\theta|\leq 1/20$, then $|w-1|<|\theta|\leq 1/20$. By \eqref{Phi-w} and \eqref{Deri-H-k}, we have
\begin{equation}\label{Phi-w-est}
|\Phi(w)|\leq \sum_{k\geq 2}4^k|C_n|(1/20)^{k-2}\leq 16|C_n|\sum_{k\geq 0}5^{-k}=20|C_n|.
\end{equation}
By \eqref{S-w} and \eqref{Phi-w-est}, it follows that
\begin{equation}
|S(w)|\geq 1-\frac{\theta^2}{1-(2n+2)|C_n|}20|C_n|\geq 1-\frac{s^{2n+1}}{n+1}\theta^2
\end{equation}
since $n\geq 2$ and $|C_n|<s^{2n+1}/(8(n+1)^2)$ by \eqref{C-n-estim}.

On the other hand, if $w=e^{i\theta}$ for $0\leq|\theta|\leq \pi$, then
\begin{equation}\label{Q-z-est}
\left|\frac{w+n}{n+1}\right|=\left(1-\frac{4n}{(n+1)^2}\sin^2\frac{\theta}{2}\right)^{1/2}\leq \left(1-\frac{4n}{\pi^2(n+1)^2}\theta^2\right)^{1/2}\leq 1-\frac{2n}{(n+1)^2\pi^2}\theta^2
\end{equation}
since $(1-x)^{1/2}\leq 1-x/2$ for $0\leq x< 1$. This means that if $w=e^{i\theta}$ for $|\theta|\leq 1/20$, then
\begin{equation}
|R(w)|\leq (1-\frac{2n}{(n+1)^2\pi^2}\theta^2)(1-\frac{s^{2n+1}}{n+1}\theta^2)^{-1}\leq 1.
\end{equation}
Moreover, $|R(w)|=1$ if and only if $w=1$.

If $w=e^{i\theta}$ for $|\theta|> 1/20$, by \eqref{S-w-ori} and Lemma \ref{para-fixed}(2), we have
\begin{equation}\label{S-w-est}
|S(w)|\geq (1-\frac{s^{2n+1}}{n+1})\prod_{i=1}^{n-1}(1-|b_i|^{2n+2})-\frac{s^{2n+1}}{3n+3}\geq 1-\frac{3s^{2n+1}}{n+1}.
\end{equation}
By \eqref{Q-z-est} and \eqref{S-w-est}, we have
\begin{equation}
|R(w)|\leq (1-\frac{2}{20^2 (n+1)\pi^2})(1-\frac{3s^{2n+1}}{n+1})^{-1}< 1.
\end{equation}
It follows that $R(w)$ maps the boundary of the unit disk into the unit disk except at $w=1$. Since $R(w)\neq \infty$ if $|w|\leq 1$, we know that $R(\overline{\mathbb{D}})\subset\mathbb{D}\cup\{1\}$. Therefore, $\widetilde{R}(\overline{\mathbb{D}})\subset\mathbb{D}\cup\{1\}$ and $\widetilde{R}$ maps $\{z\in\mathbb{C}:z^{n+1}=1\}$ onto 1.
This ends the proof of Lemma \ref{key-lemma-complex}.
\end{proof}

\vskip0.2cm

\noindent\textit{Acknowledgements.} The authors would like to thank Guizhen Cui for discussions and the referees for their careful reading and comments. The first author was supported by the National Natural Science Foundation of China under grant No.\,11271074, and the third author was supported by the National Natural Science Foundation of China under grant No.\,11231009.

%----------------------------------------------------------------------------------------------------------------

\end{document}